\documentclass[10 pt]{amsart}

\usepackage{amssymb} \usepackage{amsfonts} \usepackage{amsmath}
\usepackage{amsthm} \usepackage{epsfig} 
\usepackage{color}
\usepackage{pinlabel}
\usepackage{amscd}
\usepackage[all]{xy}
\usepackage{pinlabel}
\usepackage{tcolorbox}
\usepackage{tikz}
\usepackage{hyperref}
\usetikzlibrary{matrix,arrows,decorations.pathmorphing}

\newtheorem{lemma}{Lemma}
\newtheorem{teo}[lemma]{Theorem}
\newtheorem{prop}[lemma]{Proposition} 
\newtheorem{cor}[lemma]{Corollary} 
\newtheorem{conj}[lemma]{Conjecture}

\theoremstyle{definition}
\newtheorem{defn}[lemma]{Definition}
\newtheorem{fact}[lemma]{Fact}
\newtheorem{quest}[lemma]{Question}

\theoremstyle{remark}
\newtheorem{rem}[lemma]{Remark}

\newcommand{\Iso}{{\rm Isom}}

\newcommand{\matX}{\ensuremath {\mathbb{X}}}

\newcommand{\matR} {\ensuremath {\mathbb{R}}}

\newcommand{\matZ} {\ensuremath {\mathbb{Z}}}

\newcommand{\matH} {\ensuremath {\mathbb{H}}}

\newcommand{\matE} {\ensuremath {\mathbb{E}}}

\newcommand{\id} {\ensuremath {{\rm id}}}

\newcommand{\Isom} {\ensuremath {{\rm Isom}}}
\newcommand{\lk} {\ensuremath {{\rm link}}}

\usepackage{mathtools}
\DeclarePairedDelimiter{\set}{ \{ }{ \} }
\DeclareMathOperator{\Id}{Id}
\DeclareMathOperator{\agg}{Ad}
\DeclarePairedDelimiter{\tonde}{(}{)}
\DeclarePairedDelimiter{\quadre}{[}{]}
\DeclarePairedDelimiter{\spam}{\langle}{\rangle}

\usepackage{xfrac}

\usepackage{epstopdf}

\newcommand{\nota} [1] {\caption{\footnotesize{#1}}}

\renewcommand{\tilde}{\widetilde}

\title{Infinitesimal Rigidity for Cubulated Manifolds}
\author{Ludovico Battista}

\begin{document}

\maketitle

\begin{abstract}
    We prove the infinitesimal rigidity of some geometrically infinite hyperbolic 4- and 5-manifolds. These examples arise as infinite cyclic coverings of finite-volume hyperbolic manifolds obtained by colouring right-angled polytopes, already described in the papers \cite{BM,itamarmig}. The 5-dimensional example is diffeomorphic to $N\times \matR$ for some aspherical 4-manifold $N$ which does not admit any hyperbolic structure. To this purpose we develop a general strategy to study the infinitesimal rigidity of cyclic coverings of manifolds obtained by colouring right-angled polytopes.
\end{abstract}

\section{Introduction} \label{rigidity:section}

The deformations of hyperbolic structures form an intensively studied phenomenon. This study led to a wide variety of interesting results, among which the Hyperbolic Dehn Filling theorem in dimension 3 stands out.

The finite-volume case is the most classical one and there are several results on it. In particular it is known that hyperbolic surfaces can always be deformed, hyperbolic three-manifolds admit only non-complete deformations when non-compact and none if they are compact, and starting from dimension 4 there are no infinitesimal deformations (hence there are no deformations). 

In this article we focus on geometrically infinite (hence not finite-volume) hyperbolic manifolds. Recall that a hyperbolic manifold is called \emph{geometrically finite} if its convex core (see \cite{libroM}) has finite volume, otherwise it is called \emph{geometrically infinite}. We prove for the first time -to the best of our knowledge- the existence of rigid geometrically infinite 4- and 5-manifolds. This result is in contrast with lower dimensions: from the (now proved \cite{BrBr, B, NS}) density conjecture it follows that every geometrically infinite 3-manifold can be deformed into a geometrically finite one.

We recall some basic facts about deformations. A nice presentation of this subject can be found in \cite{CHK}.

Given a manifold with a hyperbolic metric $(M,g)$ it is possible to associate to it a holonomy $\rho$:
\[ \rho \colon \pi_1(M) \to \Isom(\matH^n) . \]
The holonomy is defined only up to inner automorphisms of $\Isom(\matH^n)$, but we often denote by $\rho$ a choice of one representative in its equivalence class.

A \emph{deformation of} $\rho$ is a smooth path of representations $\rho_t$, with $t$ varying in an interval $[0,1)$, such that $\rho_0=\rho$. Deformations of the holonomy are strictly connected with deformations of the metric $g$: when $M$ is the interior of a compact manifold, the Ehresmann-Thurston Principle (see \cite{BG}, Theorem 2.1) states that, if $\epsilon$ is small enough, for every $\rho_t$ with $ t \in [0, \epsilon)$ there is a hyperbolic metric $g_t$ on $M$ such that the associated holonomy is $\rho_t$.

The \emph{infinitesimal deformations} (we will be more precise in Section \ref{Infinitesimal_deformation:section}) are the first order solutions to the equations for the existence of deformations for a holonomy $\rho$, quotiented by the directions given by conjugations in $\Isom(\matH^n)$. We say that a manifold is \emph{infinitesimally rigid} if its infinitesimal deformations vanish. The Weil's Lemma \cite{We} states that the holonomy $\rho$ of a infinitesimally rigid manifold can be deformed only through a path of conjugations in $\Isom(\matH^n)$. The main result of this paper is the following:

\begin{teo}\label{rigidity_existence:teo}
There exist one geometrically infinite hyperbolic 4-manifold and one geometrically infinite hyperbolic 5-manifolds that are infinitesimally rigid.
\end{teo}

The examples we study are the infinite cyclic coverings of finite-volume hyperbolic manifolds described in \cite{BM, itamarmig}, see Sections \ref{dim_4:section} - \ref{dim_5:section} for the details. We briefly recall some of their nice topological properties.

Let $M$ be a compact $n$-manifold possibly with boundary with a complete finite-volume hyperbolic metric on its interior. A circle-valued Morse function on $M$ is a smooth map $f \colon M\to S^1$ such that $f|_{\partial M}$ has no critical points and $f$ has finitely many critical points, all of non degenerate type. We have
\begin{align}\label{eq:critical_points}
    \chi(M) = \sum (-1)^ic_i
\end{align}
where $c_i$ is the number of critical points of index $i$. We say that $f$ is \emph{perfect} if it has exactly $|\chi(M)|$ critical points, that is the least possible amount allowed by (\ref{eq:critical_points}).

When the dimension $n$ is odd, the Euler characteristic of $M$ vanishes, hence a perfect circle-valued Morse function is simply a fibration over $S^1$. This is never the case when $n$ is even since the Euler characteristic of $M$ never vanishes due to the Chern-Gauss-Bonnet theorem. When $n=4$, the map $f$ is perfect if and only if it has only critical points of index 2 (see \cite{BM}).

Given a perfect circle-valued Morse function $f \colon M \to S^1$, the \emph{infinite cyclic covering associated to $f$} is the smallest covering $\Tilde{M}$ of $M$ such that the following diagram commutes:
\[
\begin{tikzpicture}
  \matrix (m) [matrix of math nodes,row sep=3em,column sep=4em,minimum width=2em]
  {
     \Tilde{M} & \matR \\
     M & S^1 \\};
  \path[-stealth]
    (m-1-1) edge node [left] {} (m-2-1)
            edge node [above] {$\Tilde{f}$} (m-1-2)
    (m-2-1) edge node [above] {$f$} (m-2-2)
    (m-1-2) edge node {} (m-2-2);
\end{tikzpicture}.
\]
In other words $\Tilde{M}$ is the covering associated with $\ker(f_*)$ where $f_*$ is the map induced by $f$ on the fundamental groups.

The infinitesimally rigid manifolds of Theorem \ref{rigidity_existence:teo} are infinite cyclic coverings associated to perfect circle-valued Morse functions. Using the connection between Morse functions and handle decompositions we deduce the following:
\begin{itemize}
    \item the 4-manifold in Theorem \ref{rigidity_existence:teo} is diffeomorphic to $N \times[0,1]$ with infinitely many 2-handles attached on both sides, where $N$ is a finite-volume cusped hyperbolic 3-manifold; such a manifold was constructed in \cite{BM};
    \item the 5-manifold in Theorem \ref{rigidity_existence:teo} is diffeomorphic to $N \times \matR$, where $N$ is a non-hyperbolic aspherical 4-manifold diffeomorphic to the interior
    of a compact 4-manifold with boundary; such a manifold was constructed in \cite{itamarmig}, and the question of studying rigidity was raised in \cite{itamarmig}, Problem 29.
\end{itemize}

The perfect circle-valued Morse functions are defined on finite-volume manifolds built by colouring right-angled polytopes, with a technique already used by several authors (see \cite{KolSla} for an introduction). These manifolds are naturally homotopically-equivalent to a cube complex. In the past two years, perfect circle-valued Morse functions on such manifolds were discovered and studied in \cite{itamarmig} and \cite{BM}, following the work of Jankiewicz -- Norin -- Wise \cite{JNW} based on Bestvina -- Brady theory \cite{BB}. The cube complex structure lifts to the cyclic covering, giving a nice combinatorial description of such a infinite-volume manifold in terms of a periodic cube complex.
Our main contribution consists in finding a convenient way to implement computations on infinitesimal deformations and applying it.

All the results in this article are computer-assisted, using Sage and MATLAB. In the two cases of Theorem \ref{rigidity_existence:teo} we were able to do all the computations using symbolic calculus, ending up with rigorous results. We applied our algorithm using double-precision numbers to several other geometrically infinite hyperbolic 4-manifolds, and in all these cases we found strong numerical evidences of infinitesimal rigidity. Theoretically, it is possible to promote every numerical result to a rigorous one with the necessary amount of time and computer resources.

\subsection*{Structure of the paper} In Section \ref{colostates:section} we recall the notions of colourings and states, that were used in \cite{BM} and \cite{itamarmig} to build manifolds with convenient circle-valued functions.

In Section \ref{infinite_cyclic_cover:section} we describe the combinatorial structure of the infinite cyclic coverings of such manifolds. We used these objects to prove Theorem \ref{rigidity_existence:teo}.

In Section \ref{Infinitesimal_deformation:section} we recall some notions about infinitesimal deformations and we build some machinery to compute their dimension in our cases.

In Section \ref{numerical:section} we describe the results obtained applying our algorithm.

\subsection*{Acknowledgements} I thank Bruno Martelli for the many ideas and for Figures \ref{squares:fig} and \ref{cyc_cover:fig}, Matteo Cacciola and Leonardo Robol for the insights on the numerical analysis, Viola Giovannini and Diego Santoro for the valuable discussions. I also thank the \emph{Centro Dipartimentale di Calcolo Scientifico e Nuove Tecnologie per la Didattica} of the Dipartimento di Matematica dell'Università di Pisa for letting me use the server for the computations\footnote{The server has been acquired thanks to the support of the University of Pisa, within the call "Bando per il cofinanziamento dell’acquisto di medio/grandi attrezzature scientifiche 2016".}.

\section{Colourings and States}\label{colostates:section}

Here we recall how to use colourings and states on a hyperbolic right-angled polytope to build a hyperbolic manifold with a circle-valued function on it. Since we are going to use it intensively, we will focus on the combinatorial description, and in particular on the aspects we will use more. A detailed and comprehensive presentation can be found in \cite{BM}.

\subsection{Right-angled polytopes}
A \emph{right-angled polytope} is a polytope in $\matX^n=\matE^n$ (the euclidean flat space) or $\matH^n$ (the hyperbolic space) with all dihedral angles of value $\sfrac{\pi}{2}$.

Such an object is naturally stratified in vertices (0-faces), edges (1-faces), 2-faces, \ldots, $(n-2)$-faces, facets ($(n-1)$-faces), and one $n$-face.

Let $\Gamma< \Iso(\matX^n)$ be the group generated by the reflections along the facets of $P$. The polytope $P$ is a fundamental domain for the action of $\Gamma$. A presentation for $\Gamma$ is 
$$\langle \ r_1, \ldots, r_s \ |\ r_i^2, [r_i,r_j]\ \rangle,$$
where $r_i$ is the reflection along the $i$-th facet of $P$ and we have the relation $[r_i,r_j]=0$ whenever the $i$-th and $j$-th facet are adjacent.
Sometimes we will consider $P$ as an orbifold $P=\matX^n/_\Gamma$.

\begin{table}
\begin{center}
\begin{tabular}{c||c|ccccc}
 \!\! & Lives in & Ideal Vertices & Real vertices & Facets\\
 \hline \hline
 \!\!  Octahedron & $\matH^3$ & 6 & 0 & 8 \\
 \!\!  $24$-cell & $\matH^4$ & 24 & 0 & 24 \\
 \!\!  $P_4$ & $\matH^4$ & 5 & 5 & 10 \\
 \!\!  $120$-cell & $\matH^4$ & 0 & 600 & 120 \\
 \!\!  $P_5$ & $\matH^5$ & 10 & 16 & 16 \\
\end{tabular}
\vspace{.2 cm}
\nota{Combinatorial data of the right-angled polytopes that we use.}
\label{data:table}
\end{center}
\end{table}

We recall some interesting examples of right-angled polytopes:
\begin{itemize}
    \item In dimension 2 there are regular right-angled $k$-gons for every $k \geq 5$ in the hyperbolic plane. 
    \item In dimension 3 there are two right-angled regular polyhedra in the hyperbolic space, one with ideal vertices (the octahedron) and one compact (the dodecaheron).
    \item In dimension 4 there are two right-angled regular polytopes in the hyperbolic space, one with ideal vertices (the 24-cell) and one compact (the 120-cell).
    \item There is a family of non-regular hyperbolic right-angled polytopes $P_3, \ldots, P_8$, where $P_i$ is a polytope in $\matH^i$, very nicely described in \cite{PLV}.
\end{itemize}
The combinatorics of the ones we will use can be found in Table \ref{data:table}.

There are many hyperbolic right-angled polytopes in low dimensions. The main reason for preferring those listed in Table \ref{data:table} is that these ones are reasonably small and have a huge number of symmetries that help a lot during computations.

\subsection{Colourings}
Colouring a right-angled polytope gives a way to build a manifold. The interested reader can find a general introduction in \cite{FKS}.

Pick a polytope $P$ and choose a palette of colours $ \set{1, \ldots, c}$. We assign to each facet of $P$ a colour in such a way that adjacent facets have different colours. See Figure \ref{square_coloured:fig}.

\begin{rem}\label{rem:diff_colours_intersection}
This is equivalent to requiring that when $k$ facets meet, they all have different colours: it follows from the fact that if some facets share a common sub-face (that is not an ideal vertex), they meet pairwise (see \cite{FKS}).
\end{rem}

\begin{figure}
 \begin{center}
  \includegraphics[width = 2 cm]{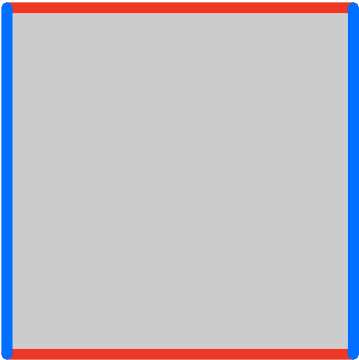}
 \end{center}
 \nota{A polytope with a colouring.}
 \label{square_coloured:fig}
\end{figure}

Using the colouring we build a manifold $M$ that is naturally tessellated into copies of $P$: we pick $2^c$ copies of $P$, denoted as $P_v$ with $v$ varying in $\tonde{\sfrac{\matZ}{2\matZ}}^c$. We then glue these polytopes along all their facets: the facet $F$ of $P_v$ is glued along the same facet of $P_{v+e_i}$ with the identity, where $i$ is the colour of $F$. See Figures \ref{torus_4_squares:fig} - \ref{torus_by_gluing:fig}.

A \emph{$k$-stratum} of $M$ is a $k$-face of any polytope $P_v$ inside $M$. Given the combinatorics of $P$, it is easy to compute the number of $k$-strata in $M$:

\begin{prop}\label{strata_number:prop}
The number of $k$-strata with $k \geq 1$ in $M$ is
\[ \#(\text{$k$-faces in $P$}) \cdot 2^{c-n+k}.\]
\end{prop}
\begin{proof}
Before gluing the facets, the number of $k$-faces in the collection $\set{P_v}$ is the number of $k$-faces of $P$ times $2^c$. When we identify the facets, the $k$-faces are glued together in groups of $2^{n-k}$. This is because every $k$-face is in the intersection of $n-k$ facets, and they all have different colours (see Remark \ref{rem:diff_colours_intersection}). We deduce that the number of $k$-strata in $M$ is:
\[ \#(\text{$k$-strata in $M$})= \frac{\#(\text{$k$-faces in $P$}) \cdot 2^c}{2^{n-k}} = \]\[ = \#(\text{$k$-faces in $P$}) \cdot 2^{c-n+k}.\]
\end{proof}

\begin{figure}
 \begin{center}
  \includegraphics[width = 4 cm]{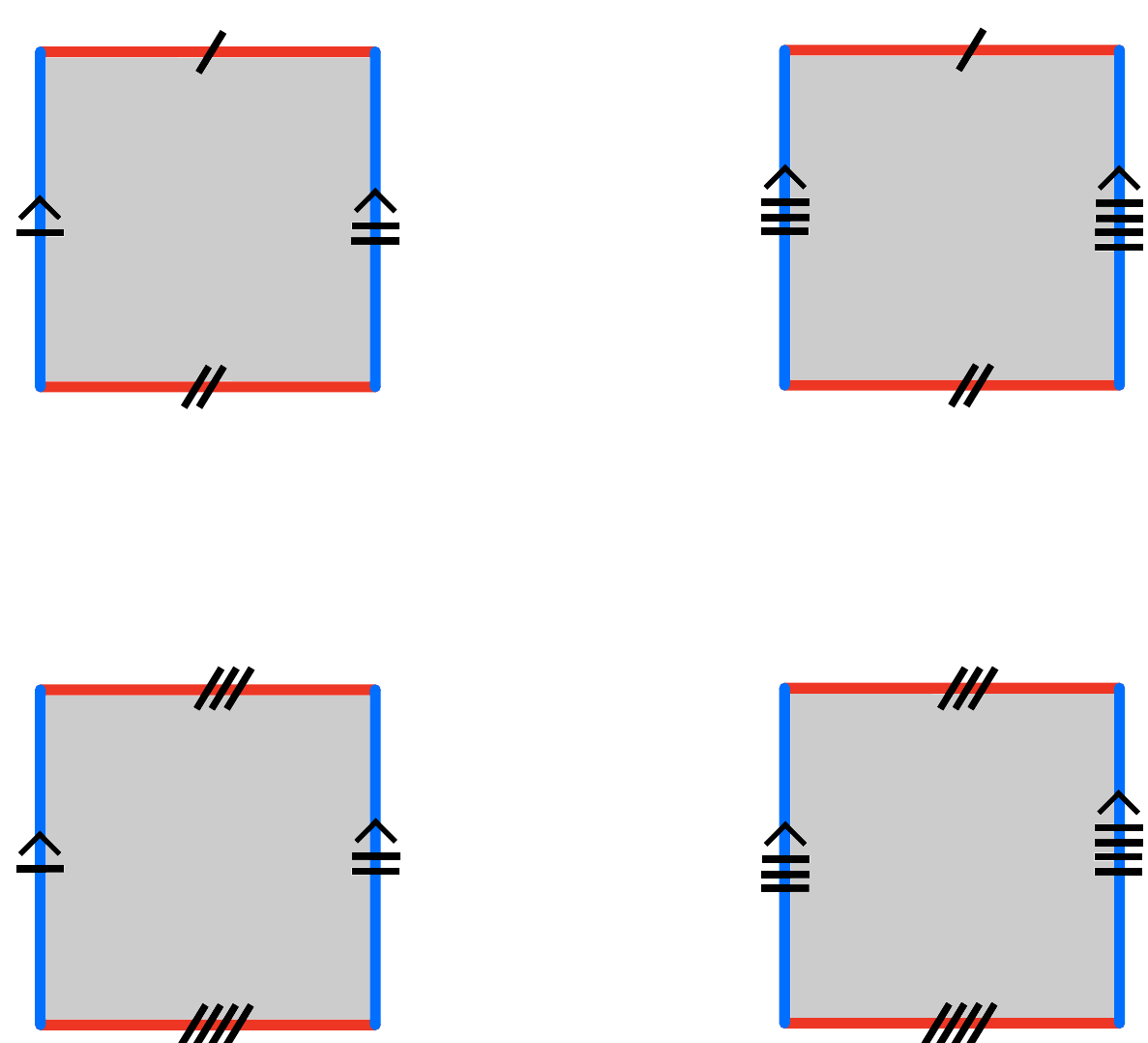}
 \end{center}
 \nota{How to build a manifold using a polytope and a colouring.}
 \label{torus_4_squares:fig}
\end{figure}

\begin{figure}
 \begin{center}
  \includegraphics[width = 4 cm]{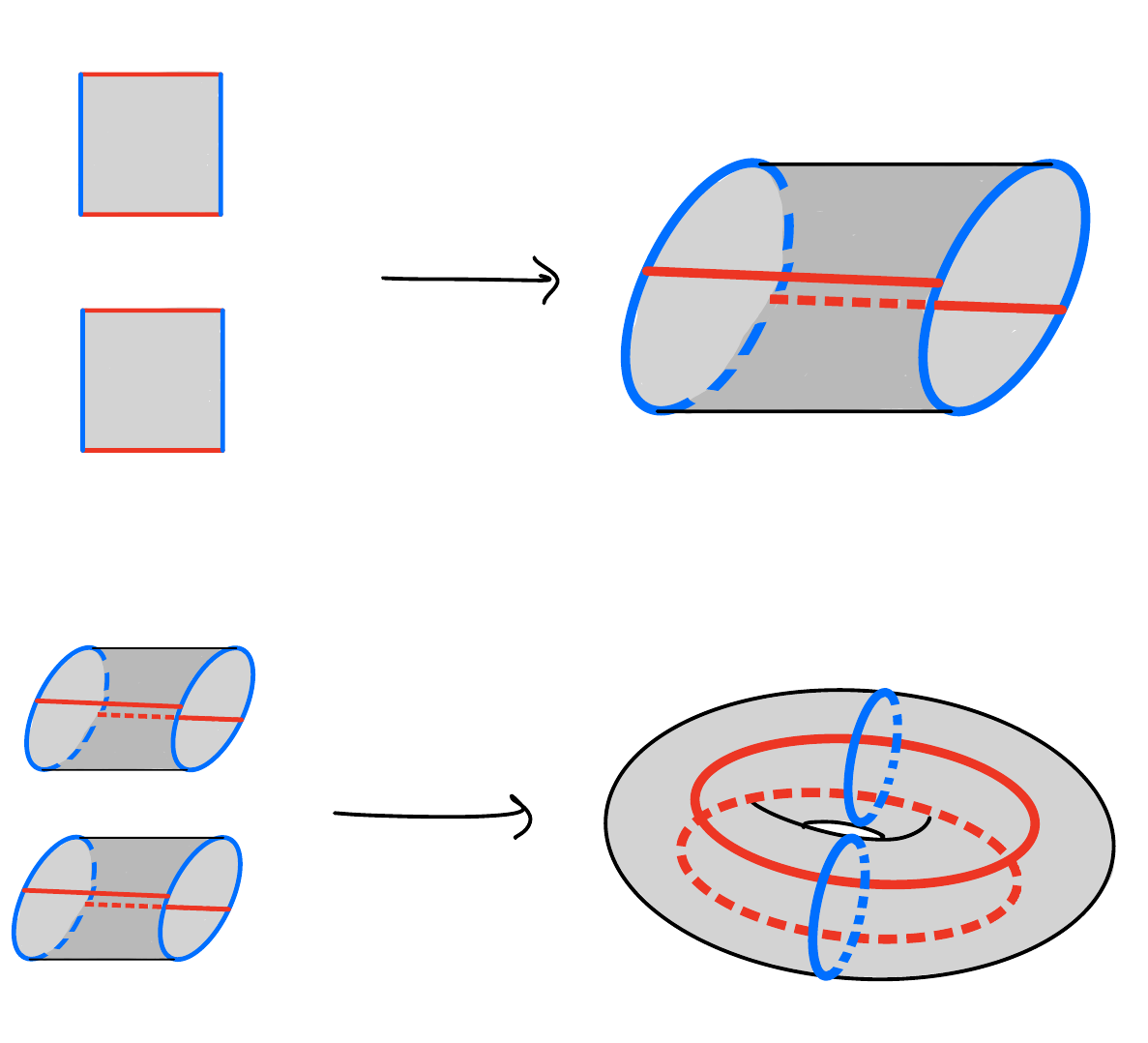}
 \end{center}
 \nota{A more visual way to identify the manifold $M$.}
 \label{torus_by_gluing:fig}
\end{figure}

\subsection{The dual cube complex}
The dual to the tessellation of $M$ is a cube complex $C$. Again, a formal definition of it can be found in \cite{BM}.

The cube complex $C$ is the core of all the constructions of the circle-valued functions. In this paper we focus on the 2-skeleton of $C$, since we are mostly interested in its fundamental group.

The cube complex $C$ can be obtained in the following way: 
\begin{itemize}
    \item we pick one vertex for each copy of the polytope $P$. We obtain $2^c$ vertices, and we denote by $x_v$ the vertex that corresponds to $P_v$, where $v$ is an element of $(\sfrac{\matZ}{2\matZ})^c$. We call an element $v \in (\sfrac{\matZ}{2\matZ})^c$ \emph{even} (\emph{odd}) if $\sum_{i=1}^c v_i$ is even (odd). A vertex $x_v$ is \emph{even} (\emph{odd}) if $v$ is even (odd);
    \item we add one edge between the vertices $x_v$ and $x_w$ for each facet separating $P_v$ and $P_w$. Notice that $v$ and $w$ differ only in one coordinate, so exactly one of them is even: one useful way to list all the edges is by considering all the facets of the polytopes $P_v$ where $v$ is even;
    \item we add one square with vertices $x_{v^1}$, $x_{v^2}$, $x_{v^3}$, and $x_{v^4}$ whenever $P_{v^1}$, $P_{v^2}$, $P_{v^3}$, and $P_{v^4}$ share a common codimension 2 face. Notice that $v^1$, $v^2$, $v^3$, and $v^4$ must be of the form $v$, $v+e_i$, $v+e_j$, and $v+e_i+e_j$. The edges of such a square can be deduced from the parity of the vertices $x_{v^i}$: there is an edge between vertices with different parity;
    \item we go on with all other faces.
\end{itemize}   

The number of $k$-cubes in $C$ is exactly the number of $(n-k)$-strata of the tessellation of $M$. See Figure \ref{dual_cube_complex:fig}.

\begin{figure}
 \begin{center}
  \includegraphics[width = 4 cm]{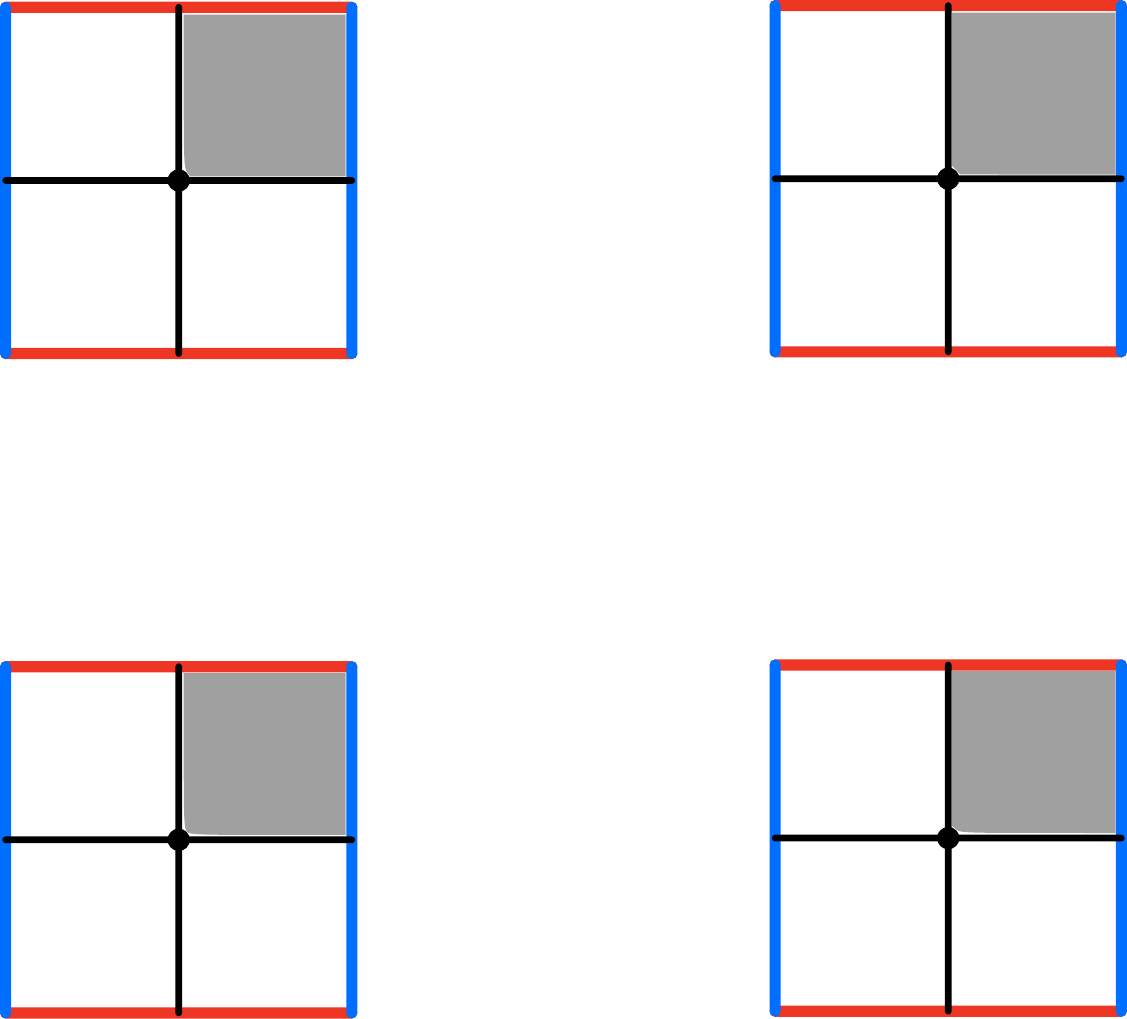}
 \end{center}
 \nota{The cube complex dual to the tessellation. Only one of the four squares is highlighted.}
 \label{dual_cube_complex:fig}
\end{figure}

When $P$ has no ideal vertices, the underlying space of $C$ is homeomorphic to $M$. When there are ideal vertices, the manifold $M$ deformation retracts onto $C$. In any case, the homotopy type of $M$ and $C$ is the same. 

We are interested in the holonomy of a covering of $M$:
\[ \Tilde{\rho} \colon \Tilde{M} \to \Isom(\matH^n). \]
The cube complex structure $C$ of $M$ will lift to a cube complex structure $\Tilde{C}$ on $\Tilde{M}$.

In the works \cite{itamarmig} and \cite{BM} the cube complex $C$ has been enlarged to a bigger cube complex when there are ideal vertices. This is because the authors wanted to define a circle-valued function on $M$. Here, we do not need to put any attention on this aspect, since only the homotopy type of the maps will matter.

\subsection{States}
A \emph{state} gives a way to define a map from $M$ to $S^1$. This notion was introduced in \cite{JNW}. We work on $C$ instead of $M$, disregarding the fact that $C$ is only a deformation retract of $M$ when $P$ has ideal vertices.

We want to build a map from $C$ to $S^1$. We send all the vertices to $1 \in S^1$. We then orient each edge. We use the orientation to identify each edge with the standard interval $I=[0,1]$, on which we have a natural function to $S^1$ (see Figure \ref{map_on_edges:fig}). While doing this, we want to make sure we will be able to extend the map on the 2-skeleton: every square of the 2-skeleton describes one obstruction to the extension, see Figure \ref{maps_on_squares:fig}. In particular, we need the image of the boundary of the square to be trivial in the fundamental group of $S^1$. If we manage to extend the function on the 2-skeleton, then the map can be defined on the whole $C$ (the $k$-skeleta with $k \geq 3$ do not provide any obstruction because $S^1$ is aspherical).

\begin{figure}
 \begin{center}
  \includegraphics[width = 6 cm]{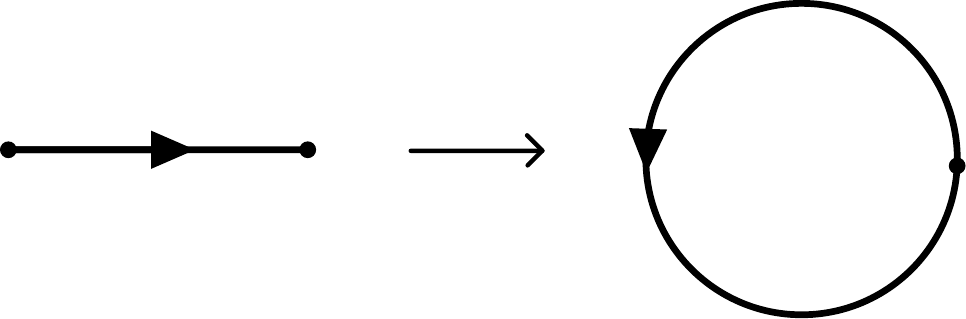}
 \end{center}
 \nota{Given an orientation on an edge, it is possible to identify it with the standard interval $[0,1]$, that is sent to $S^1$ by the quotient $\sfrac{[0,1]}{\set{0;1}}$.}
 \label{map_on_edges:fig}
\end{figure}

\begin{figure}
 \begin{center}
  \includegraphics[width = 10 cm]{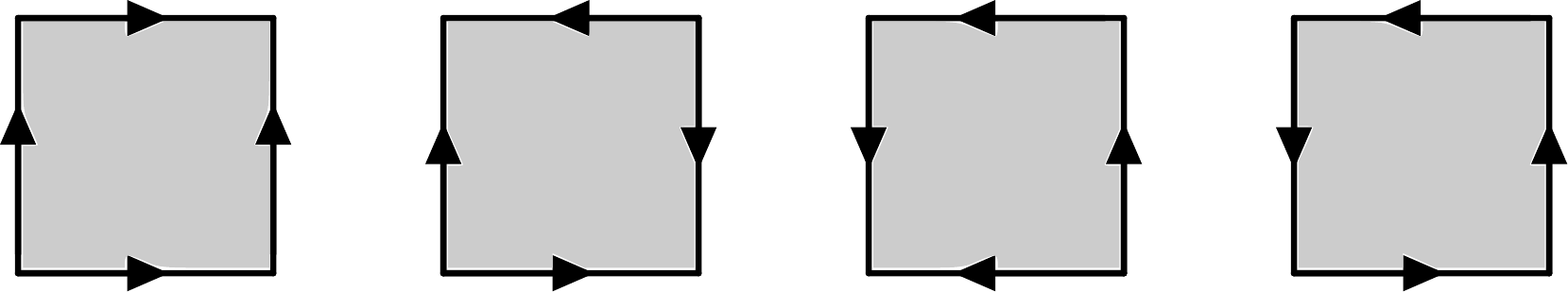}
 \end{center}
 \nota{The four possible maps that can be defined on the boundary of a square, up to equivalence. The first two can be extended on the interior of the square, the last two do not. We call the first square \emph{coherently oriented}, and the second one \emph{bad}.}
 \label{maps_on_squares:fig}
\end{figure}

To do this we use the notion of \emph{state} (see Figure \ref{square_state:fig}):
\begin{defn}
A \emph{state} of a polytope $P$ is the assignment of the letter "I" or the letter "O" to each of its facets.
\end{defn}
\begin{figure}
 \begin{center}
  \includegraphics[width = 2 cm]{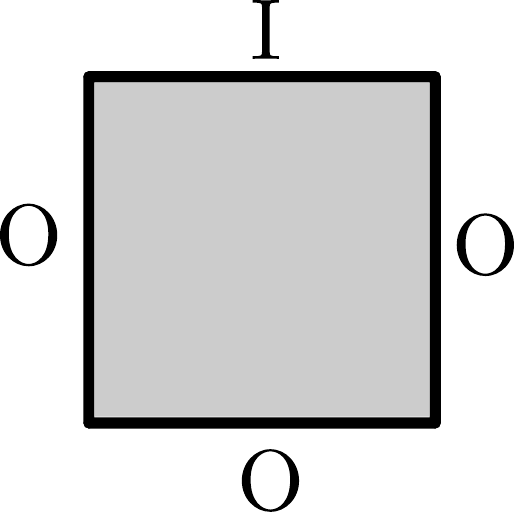}
 \end{center}
 \nota{A polytope with a state.}
 \label{square_state:fig}
\end{figure}

We choose a state $s_0$ on $P_0$. Consider the edges of $C$ with one endpoint in $x_0$. Each of these edges is dual to a facet of $P_0$. We orient an edge outward (inward) with respect to $x_0$ if the corresponding dual facet has letter "O" ("I"). We now want to orient all the other edges of $C$ making sure that we find a function that can be extended on the 2-skeleton. There are two ways of doing this, one used in \cite{BM} and one used in \cite{itamarmig}.

The first one is the following: we define the state $s_{e_{i_1}+ \ldots +e_{i_r}}$ to be $s_{0}$ where we swapped the letters of the facets that have colours in $\set{i_1, \ldots, i_r}$. Each $s_v$ defines an orientation on each edge with one endpoint in $x_v$. It is easy to check that this definition is well-posed. Furthermore, we can extend this map on the 2-skeleton. Suppose we have a square: it is the dual of the intersection of two facets of $P_v$ that have colours $i$ and $j$. The vertices of the cube are $v$, $v + e_i$, $v + e_j$, and $v + e_i + e_j$. Following the recipe, it is easy to see that the orientation of the edge $(v, v+e_i)$ is the same as the orientation $(v+e_j, v+e_j+e_i)$. The same holds for $(v, v+e_j)$ and $(v+e_i, v+e_i+e_j)$. Hence, opposite sides of the square have the same orientation (in this case, we say that the edges are oriented \emph{coherently}, see Figure \ref{maps_on_squares:fig}). This ensures that the map is homotopically trivial on the boundary of the square, hence it can be extended.

The second way comes from the following consideration: for a map defined on the boundary of the square to be trivial in homotopy, we do not really need opposite edges to be oriented coherently. There is also another possibility, as shown in Figure \ref{maps_on_squares:fig}. We call squares of this type \emph{bad}. In \cite{itamarmig} the authors partition the colours in disjoint pairs. The state $s_{v+e_i}$ is defined as the state $s_v$ where the letters of the facets with colour $i$ and the colour paired with $i$ are switched from "O" to "I" and viceversa. In this case we talk about \emph{paired colours}. To ensure that this process will produce a map on the 1-skeleton that can be extended on the 2-skeleton one has to check that every square will be either oriented coherently or a bad square.

\begin{rem}
In the papers \cite{BM, itamarmig}, the authors exploit some methods to control the critical points of the maps obtained. In particular, they exhibit perfect circle-valued Morse functions in dimension 4 (in \cite{BM}) and 5 (in \cite{itamarmig}).
\end{rem}

There is a convenient way to extend the map to the whole $C$. It is called \emph{diagonal map}. This is used to describe precisely the fiber of the map. We describe it rapidly.

\subsection{The diagonal map}\label{diag_map:subsection} We start with a definition:
\begin{defn}
A cube with oriented edges is \emph{coherently oriented} if every square inside it is coherently oriented. A cube complex with an orientation on the edges is \emph{coherently oriented} if every cube inside it is coherently oriented.
\end{defn}
When $C$ is coherently oriented, one can identify every cube with the standard cube in $\matR^n$ with the edges oriented as going outside from the origin (up to permutation of coordinates). On the standard cube $[0,1]^d$ we can define the map with values is $\sfrac{\matR}{\matZ} = S^1$:
\[ (x_1,\ldots,x_d) \mapsto \quadre*{\sum_{i=1}^d x_i}.\]
These maps extend the maps defined on the edges and glue together to a well-defined map on $C$ that we call \emph{the diagonal map}. 

We need some attention in the case we have some squares that are not oriented coherently. 
\begin{defn}
Let $Q_1$ and $Q_2$ be two cubes with oriented edges. The \emph{orientation induced on the edges of $Q_1 \times Q_2$} is the only one such that both projections preserve the orientation of the edges (see Figure \ref{induced_orientation:fig}).
\end{defn}

\begin{figure}
 \begin{center}
  \includegraphics[width = 3 cm]{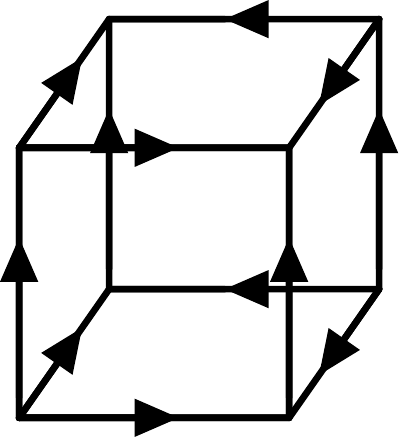}
 \end{center}
 \nota{The orientation induced on a 3-dimensional cube by a bad square times an interval.}
 \label{induced_orientation:fig}
\end{figure}

An useful property of an orientation on the edges of a cube is the following:
\begin{defn}\label{quasi_coherent:def}
A $d$-cube with oriented edges is \emph{quasi-coherently oriented} if one of the following holds:
\begin{itemize}
    \item it is coherently oriented;
    \item it has the orientation induced by a bad square times a coherently oriented $(d-2)$-cube.
\end{itemize}
A cube complex with an orientation on the edges is \emph{quasi-coherently oriented} if every cube inside it is quasi-coherently oriented.
\end{defn}

Suppose that we have a quasi-coherent orientation on $C$ (this is the case in \cite{itamarmig}). On the coherently oriented cubes we use the same map as before. On the other type of cubes, we divide the bad square $Q'$ in four triangles $T_i$ as in Figure \ref{square_in_triangles:fig}. We identify each triangle with a standard one with vertices (0,0), (1,0) and $(\frac12, \frac12)$, and we consider on it the projection $p$:
\[ (y_1, y_2) \mapsto y_1. \]
We then divide the cube in prisms $T_i \times Q$, whose factors are identified with the standard triangle and the standard $(d-2)$-cube. On every prism we can define a map in a way very similar to the previous case:
\[ ((y_1,y_2),(x_1,\ldots,x_{d-2})) \mapsto \quadre*{y_1 + \sum_{i=1}^d x_i}.\]

\begin{figure}
 \begin{center}
  \includegraphics[width = 8 cm]{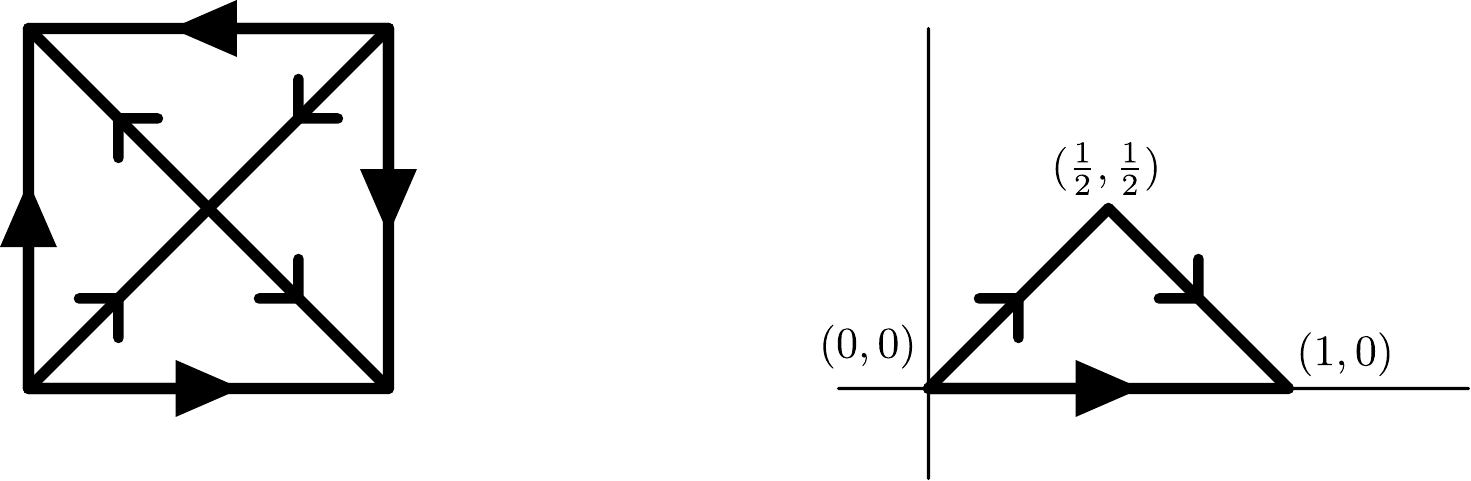}
 \end{center}
 \nota{We divide a bad square in four triangles (left). We can identify each of them with the standard one respecting the arrows (right).}
 \label{square_in_triangles:fig}
\end{figure}

Also in this case these maps extend the maps defined on the edges and glue together to a well-defined map on $C$ that we still call \emph{the diagonal map}. 

\section{The Infinite Cyclic Covering} \label{infinite_cyclic_cover:section}

In this article we focus on the infinite cyclic covering associated with the maps built in \cite{BM, itamarmig}. Here we describe in detail the combinatorial structure of these infinite cyclic coverings. 

\subsection{The Cube Complex structure}
Let $f \colon C \to S^1$ be a homotopically non-trivial map. The map $f$ induces a map $f_*$ on the fundamental groups. Since $M$ retracts on $C$, we can identify their fundamental groups. The \emph{infinite cyclic covering of $M$ (resp. $C$) associated to $f$} is the covering of $M$ (resp $C$) associated with $\ker(f_*)$, and we denote it by $\Tilde{M}$ (resp. $\Tilde{C}$).

The tessellation of $M$ into polytopes $P_v$ lifts to a tessellation for $\widetilde M$ into polytopes $P_v^t$, parametrized by $v \in \tonde{\sfrac{\matZ}{2\matZ}}^c$ and $t\in \matZ$, with the requirement that $v_1+\cdots + v_c + t$ is even. 

The facet $F$ of $P_v^t$ is identified with the identity map to the corresponding facet of $P_{v+e_i}^{t \pm 1}$, where $i$ is the colour of $F$ and the sign $+1$ or $-1$ depends on whether the status of $F$ in $P_v$ is O or I. 

The covering $\widetilde M \to M$ is the forgetful map $P_v^t \to P_v$. The monodromy of the covering $\widetilde M \to M$ is the map $\tau\colon \widetilde M\to \widetilde M$ that sends $P_v^t$ to $P_v^{t+2}$ identically.

\begin{figure}
 \begin{center}
  \includegraphics[width = 12.5 cm]{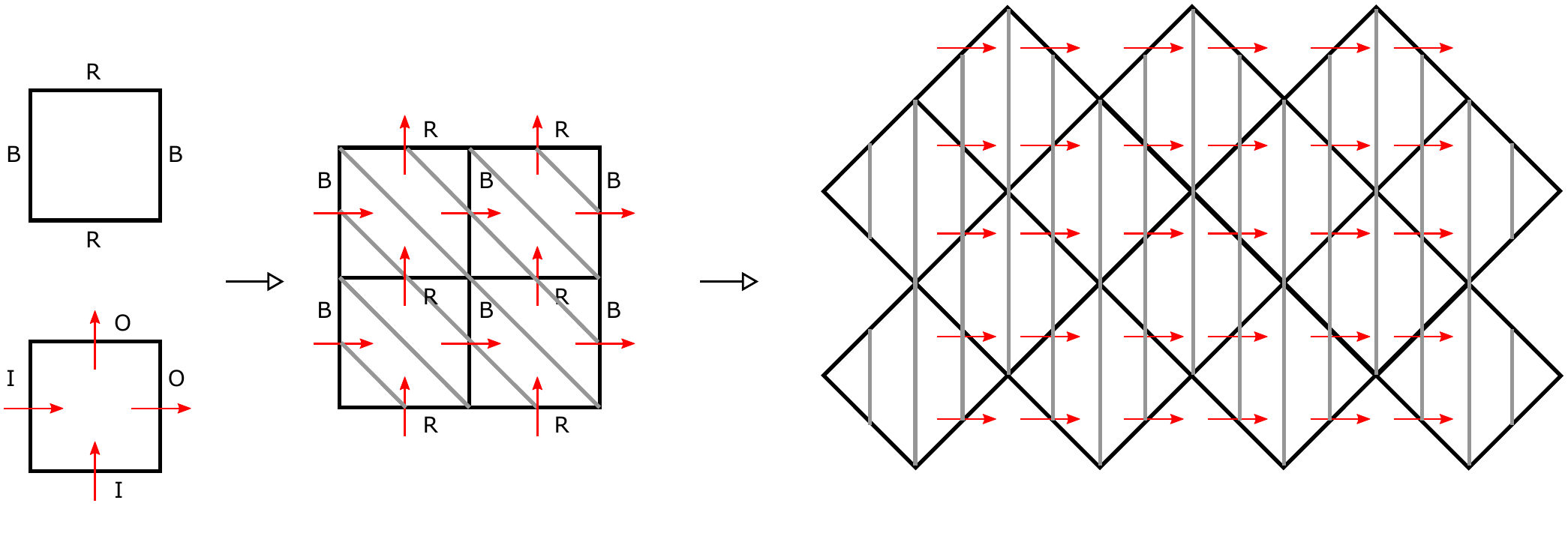}
 \end{center}
\nota{The manifold $\widetilde M$ when $P$ is a square with the colours and state indicated on the left. The central figure shows the torus $M$ tessellated into four squares (the opposite external vertical and horizontal edges should be identified), and the fibration $f\colon M\to S^1$ is suggested by the red arrows, with some fibers drawn in grey. The manifold $\widetilde M$ is the infinite annulus shown on the right (the zigzagged top and bottom edges should be identified), and the fibration $\tilde f\colon \widetilde M \to \matR$ is also suggested by the red arrows, with some fibers drawn in grey (so $\tilde f$ is the projection on the horizontal axis).}
 \label{squares:fig}
\end{figure}

The dual of this tessellation of $\widetilde M$ is $\Tilde{C}$, onto wich $\Tilde{M}$ retracts. The vertices of $\widetilde C$ are parametrized by the pairs $(v,t)$ with $v\in \tonde*{\sfrac{\matZ}{2 \matZ}}^c, t\in \matZ$ and $v_1+\cdots+v_c+t$ even. The vertex $(v,t)$ is dual to $P_v^t$.

The lifted map $\tilde f\colon \widetilde C \to \matR$ sends the vertex $(v,t)$ of $\widetilde C$ to $t$ and is extended diagonally on $\widetilde C$. See an instructing example with the square in Figure \ref{squares:fig}. 



The \emph{level} of a vertex $(v,t)$ of $\widetilde C$ is the number $t$, that is its image along $\tilde f$. More generally, the \emph{level} of a $k$-cube in $\widetilde C$ is the average of the levels of its vertices, that is the image of its center along $\tilde f$. When the states are obtained as in \cite{BM}, this number lies in $\matZ$ or in $\matZ + \frac 12$ according to the parity of $k$. When we admit paired colours (as in \cite{itamarmig}), this is no longer true: for example, bad squares have level in $\matZ + \frac 12$. The level of every $k$-stratum of the tessellation into polytopes $\{P_v^t\}$ is by definition the level of its dual $(n-k)$-cube.

We are interested in particular in the 2-skeleton of $\Tilde{C}$. Every facet $F$ of the tessellation of $\Tilde{M}$ is adjacent to $P_v^t$ and $P_{v+e_i}^{t+1}$ and has level $t+\frac 12$. Facets correspond to the edges of the cube complex. Given a codimension-2 face, the corresponding dual square can be of two types: if the edges are oriented coherently, then its lifts are adjacent to four polytopes
$$P_v^{t-1}, P_{v+e_i}^t, P_{v+e_j}^t, P_{v+e_i+e_j}^{t+1}$$
and have level $t$.
If it is a bad square (remember Figure \ref{maps_on_squares:fig}), then its lifts are adjacent to four polytopes
$$P_v^{t}, P_{v+e_i}^{t+1}, P_{v+e_j}^{t+1}, P_{v+e_i+e_j}^t$$
and have level $t+\frac12$.

\subsection{The fiber}\label{fibers:subsection}

The map $f \colon C \to S^1$ lifts to a map $\Tilde{f} \colon \Tilde{C} \to \matR$. We are especially interested in the fiber $C^{\rm sing}_t = \tilde f^{-1}(t)$ for some $t \in \matZ$. We call it \emph{singular fiber} following the notation of \cite{BM}; see Figure \ref{fiber_in_cubes:fig}.

\begin{figure}
 \begin{center}
  \includegraphics[width = 8cm]{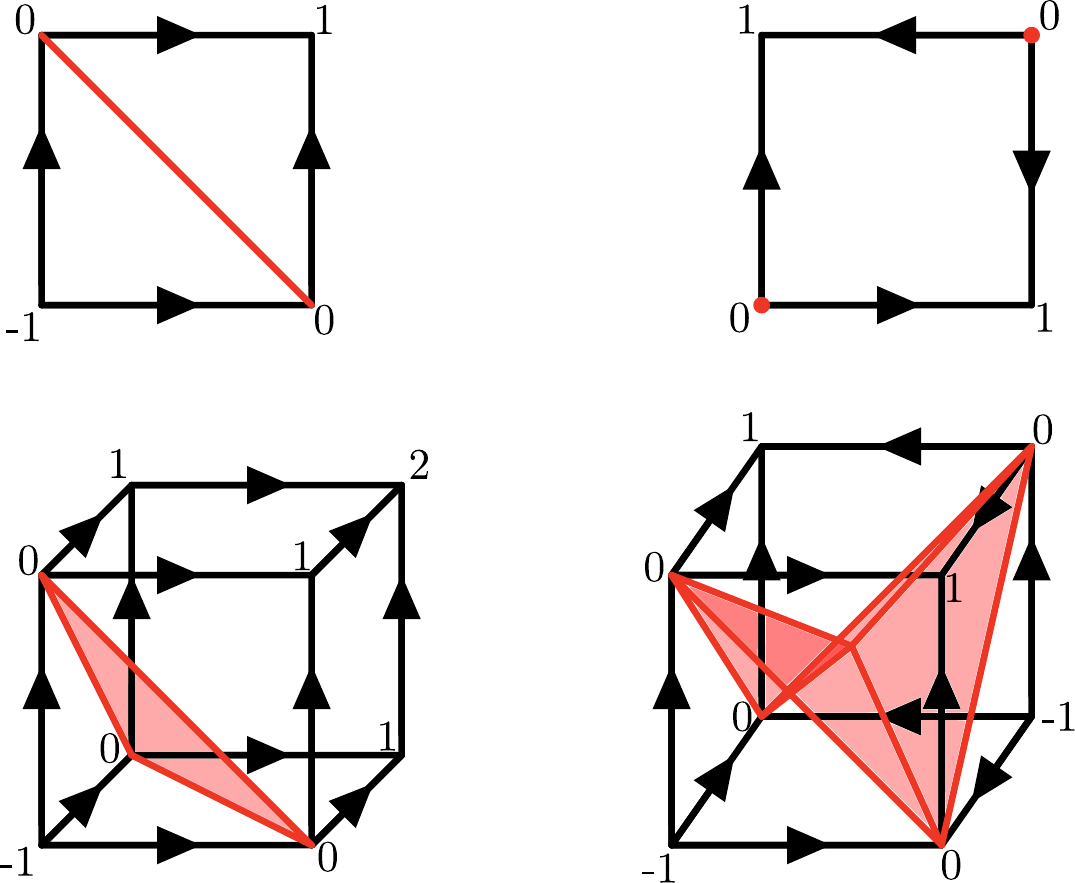}
 \end{center}
 \nota{The singular fiber $\Tilde{f}^{-1}(0)$ inside the cubes of the cube complex $\Tilde{C}$. In the top left in a coherently oriented square, in the top right in a bad square (where the intersection is only 2 vertices), in the bottom left in a coherently oriented 3-cube, in the bottom right in a bad square times an interval. The numbers denote the levels of the vertices in $\Tilde{C}$.}
 \label{fiber_in_cubes:fig}
\end{figure}

We describe a condition introduced in \cite{BB} that assures that the inclusion of $C^{\rm sing}_t$ in $\Tilde{C}$ is $\pi_1$-surjective. See also \cite{BM,itamarmig} for more details.

Let $C$ be a cube complex equipped with an orientation on its edges, and $f\colon C \to S^1$ be the diagonal map. Let $v$ be a vertex of $C$. Let $\lk(v)$ be the link of $v$ in $C$. By construction $\lk(v)$ is an abstract simplicial complex. Every vertex of $\lk(v)$ indicates an edge incident to $v$, and we assign to it the \emph{status} I (In) or O (Out) according to whether the edge points towards $v$ or away from $v$.

Following \cite{BB}, we define the \emph{ascending link} $\lk_\uparrow(v)$ (respectively, \emph{descending link} $\lk_\downarrow(v)$) to be the subcomplex of $\lk(v)$ generated by all the vertices with status O (respectively, I).

The following result is proved in \cite{BB}, Corollary 2.6:

\begin{fact}\label{inclusion_pi1_surjective:fact}
If the descending and ascending link of every vertex are connected, the inclusion of $C^{\rm sing}_t$ in $\Tilde{C}$ is $\pi_1$-surjective.
\end{fact}

All the orientations of the edges of $C$ that we consider in this paper satisfy the hypothesis of Fact \ref{inclusion_pi1_surjective:fact}.

\begin{figure}
 \begin{center}
  \includegraphics[width = 8cm]{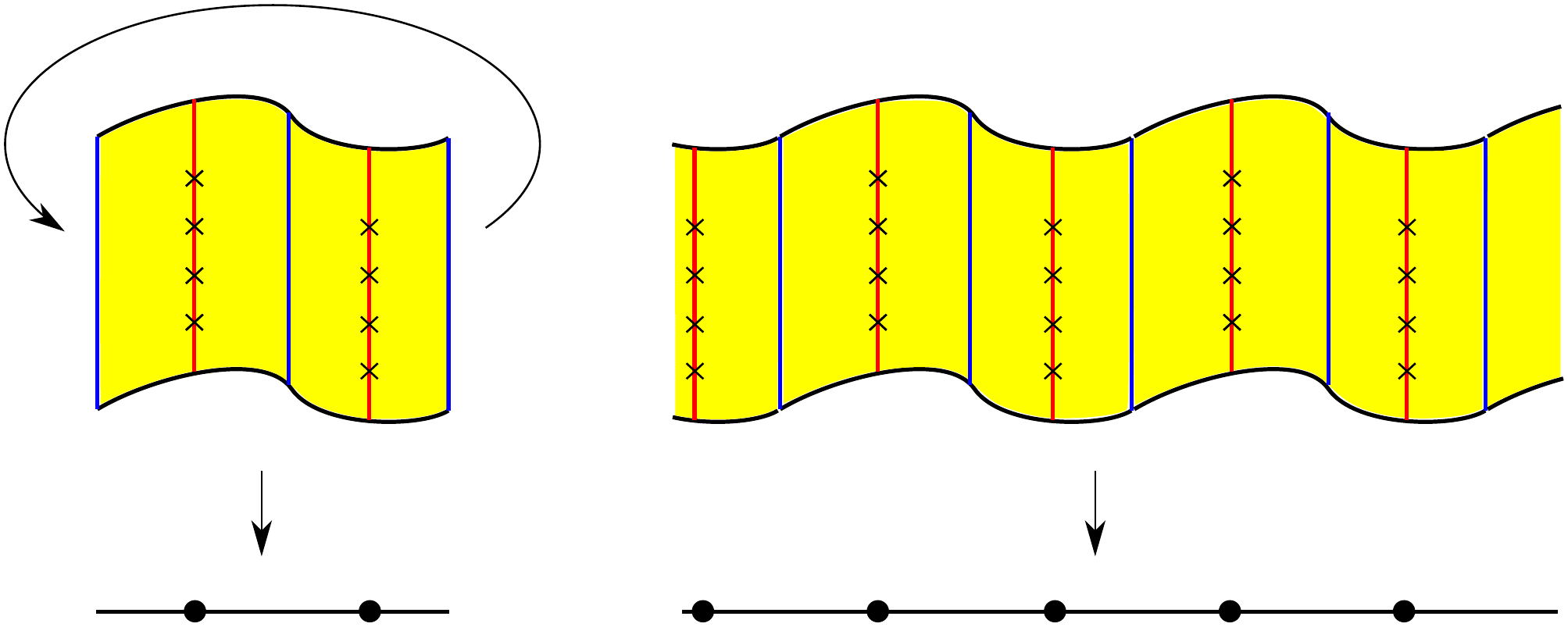}
 \end{center}
 \nota{A circle-valued Morse function (left) and its lift (right). The critical points are indicated with an X. The regular fiber is the blue line, the singular fiber is the red line. Every critical point of index $i$ corresponds to the attachment of a $i$-handle.}
 \label{cyc_cover:fig}
\end{figure}

\subsection{The finite subcomplex $F_{[m,n]}$}\label{Fnm:subsection}
The cube complex structure is a nice combinatorial structure that can help in finding the space of infinitesimal deformations of $\Tilde{M}$. The main problem with $\Tilde{C}$ is the fact that it is infinite, hence we cannot apply any finite algorithm to it. Here we define a nice finite subcomplex $F_{[m,n]}$ of $\Tilde{C}$ that contains all the information we need, \emph{i.e.}\ its inclusion is $\pi_1$- surjective.

Recall that the map $\Tilde{f}$ sends every vertex in $\widetilde C$ to some integer $m \in \matZ$, every edge to an interval $[m,m+1]$, and every square either to an interval $[m,m+1]$ (if it is the lift of a bad square) or to an interval $[m,m+2]$ (if it is the lift of a coherently oriented square). For every pair $m < n$ of integers we define $F_{[m,n]}$ as the union of all the (finitely many) squares whose image lies in $[m,n]$. We want to prove that under certain conditions its inclusion in $\Tilde{C}$ is $\pi_1$-surjective. We start with a lemma:
\begin{lemma}\label{zigzag:lemma}
Let $G$ be a quasi-coherently oriented cube (see Definition \ref{quasi_coherent:def}) inside $C$ with $\dim(G) \leq 9$. Let $\Tilde{f} \colon G \to \matR$ be a lift of the diagonal map on $G$ that takes the value $0$. Let $\gamma \colon [0,1] \to G$  be a path such that $\gamma(0)$ and $\gamma(1)$ are vertices of $G$ that are sent to $0$ via $\Tilde{f}$. Let $E$ be the union of the edges of $G$ whose endpoints have image through $\Tilde{f}$ in the set $\set{-1;0;1}$. Then $\gamma$ is fixed-endpoint homotopic to a path $\gamma'$ contained in $E$.
\end{lemma}
\begin{proof}
We do not know whether the hypothesis on the dimension of $G$ is necessary. In any case, we will need this result only in dimension less or equal than $5$.

We start by noticing that $G$ is contractible, hence two paths are fixed-endpoint homotopic if and only if they have the same extremal points. We also notice that every vertex $v$ such that $\Tilde{f}(v)=0$ is contained in $E$, because the values of $\Tilde{f}$ on the two endpoints of any edge differ by 1. Therefore to complete the proof we just need to check that the subcomplex $E$ is connected. To do this we use the script "Check\_zigzag" written in Sage available at \cite{code}.
\end{proof}

We are now ready to prove the following:
\begin{prop}
Suppose that $C$ is quasi-coherently oriented and $\dim(M) \leq 9$. Let $n-m \geq 2$ and let $x$ be a vertex of $\Tilde{C}$ of level $m+1$. Inside $\pi_1(\Tilde{C},x)$, the subgroup $i_*(\pi_1(F_{[m,n]},x))$ contains the subgroup $i_*(\pi_1(\Tilde{f}^{-1}(m+1),x))$.
\end{prop}
\begin{proof}
We prove that $i_*(\pi_1(F_{[-1,1]},x))$ contains the subgroup $i_*(\pi_1(\Tilde{f}^{-1}(0),x))$. The general case is a straightforward generalization.

Consider a loop $\alpha \colon [0,1] \to \Tilde{f}^{-1}(0)$ with $\alpha(0)=\alpha(1)=x$. Up to homotopy inside $\Tilde{f}^{-1}(0)$, we can suppose that $\alpha$ is the concatenation of a finite number of paths\footnote{secondo te devo giustificarlo meglio?}:
\[ \alpha = \alpha_1 \cdot \ldots \cdot \alpha_n \]
where the image of each $\alpha_i$ is contained in a cube $G_i$ inside $\Tilde{C}$.
For each $i=1,\ldots,n-1$, the point $v_i=\alpha_i(1)=\alpha_{i+1}(0)$ is contained in a cube $J_i$ contained in the intersection $G_i \cap G_{i+1}$. Now we notice that, by the definition of the diagonal map, the fact that $\Tilde{f}$ takes the value $0$ on the cube $J_i$ implies that there is a vertex $w_i$ of $J_i$ such that $\Tilde{f}(w_i)=0$. Let $\beta_i$ be any path inside $J_i$ that connects the points $v_i$ and $w_i$. For convenience, let $\beta_0=\beta_{n}$ be the constant path that takes value $x=\alpha(0)=\alpha(1)$. Consider now
\[ \gamma =  \prod_{i=0}^{n-1} \beta_i ^{-1} \cdot \alpha_{i+1} \cdot \beta_{i+1} .\]
The loop $\gamma$ is homotopic to $\alpha$, because the compositions $\beta_i \cdot \beta_i^{-1} $ cancel out. Moreover, each term of the form $\beta_i ^{-1} \cdot \alpha_{i+1} \cdot \beta_{i+1}$ is a path inside $G_{i+1}$ whose endpoints are sent to $0$ by $\Tilde{f}$, see Figure \ref{cammino_tra_0:fig}. We can then apply the Lemma \ref{zigzag:lemma} and conclude that $\gamma$ is homotopic to a composition of paths contained in $F_{[-1,1]}$. This concludes the proof.
\begin{figure}
 \begin{center}
  \includegraphics[width = 4cm]{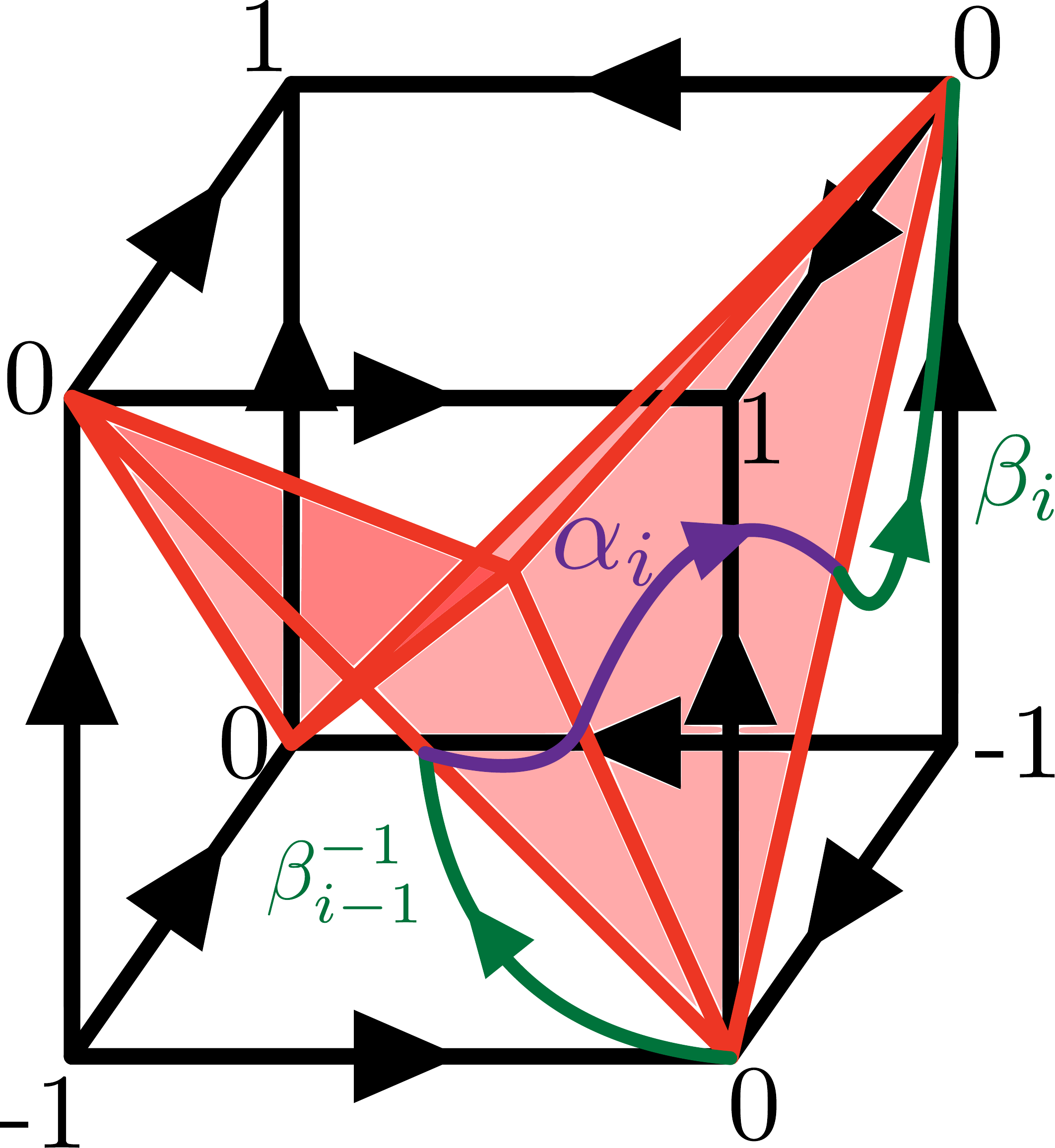}
 \end{center}
 \nota{The concatenation $\beta_{i-1} ^{-1} \cdot \alpha_{i} \cdot \beta_i$ is a path inside $G_i$ that connects two vertices that are sent to 0 via $\Tilde{f}$.}
 \label{cammino_tra_0:fig}
\end{figure}
\end{proof}

Using this proposition together with Fact \ref{inclusion_pi1_surjective:fact} we deduce the following:

\begin{cor}\label{Fnmsurj:cor}
Suppose that $C$ is quasi-coherently oriented, that $\dim(M)<10$ and that all the ascending and descending links in $C$ are connected. If $n-m \geq 2$, the map induced by the inclusion $i \colon F_{[m,n]} \to \Tilde{M}$ on the fundamental groups
\[i_* \colon \pi_1(F_{[m,n]}) \twoheadrightarrow \pi_1(\Tilde{M}) \]
is surjective.
\end{cor}

\section{The space of infinitesimal deformations of a representation}\label{Infinitesimal_deformation:section}
We recall some standard facts on deformations of hyperbolic structures. These can be found for instance in \cite{KS}.

Let $X$ be a hyperbolic $n$-manifold and $\rho$ a representation of $\pi_1(X)$ in the Lie group $G={\rm O}^+(n,1)=\Isom(\matH^n)$. The main example we keep in mind is when $\rho$ is the holonomy of $X$. Recall that a \emph{deformation} of $\rho$ is a smooth path of representations $\rho_t$, with $t$ varying in an interval $[0,1)$, such that $\rho_0=\rho$. 
For each element $g$ in $\pi_1(X)$ we consider the initial deformation direction: 
\[ \frac{d}{dt} \rho_t(g) |_{t=0}.  \]
In this way we assign to any element $g$ of $\pi_1(X)$ a vector $\zeta(g)$ in the tangent space of $\rho(g)$ in $G$. From the Leibnitz rule we deduce the following equality
$$\zeta(gh) = \zeta(g) \rho (h) + \rho(g) \zeta(h).$$
In the literature the element $\zeta(g)$ is often moved to $\mathfrak{g}$ using the differential of the right-multiplication by $\rho(g)^{-1}$, that is itself the right multiplication by $\rho(g)^{-1}$ since $G$ is a matrix group. So we define
$z(g) = \zeta(g) \rho(g)^{-1}$
and obtain a map
\[ z : \pi_1(X) \longrightarrow \mathfrak{g}. \]
To switch between $\zeta$ and $z$ it is sufficient to right-multiply by $\rho(g)$ or by $\rho(g)^{-1}$. The Leibnitz rule for $\zeta$ transforms into the \emph{cocycle condition}
\begin{equation} \label{cocycle:eqn}
z(gh) = z(g) + \rho(g) z(h) \rho (g)^{-1}.
\end{equation}
A \emph{cocycle} is a map $z \colon \pi_1(X) \to \mathfrak{g}$ that satisfies the cocycle condition. The cocycles form a vector space denoted by
\[Z^1(\pi_1(X), \mathfrak{g}_{\agg \rho}). \]
This space contains the directions along which there could be a chance to deform the representation  $\rho$. Some of these directions are quite obvious and we would like to ignore them: consider a smooth path $g_t$ of elements in $G$ such that $g_0=\Id$, and let $k_t$ be the conjugation by the element $g_t$. It is possible to deform the representation $\rho$ as $\rho_t = k_t(\rho)$. The resulting deformation direction $z$ only depends on the tangent vector $V=\frac{d}{dt}g_t|_{t=0} \in \mathfrak{g}$. Explicitly, we get
$$\zeta(g) = V \rho(g) - \rho(g) V, \qquad z(g)= V - \rho(g)V\rho(g)^{-1}.$$ 
A cocycle obtained in this way is called a \emph{coboundary}, and the subspace of all coboundaries is denoted by $ B^1(\pi_1(X), \mathfrak{g}_{\agg \rho}) $. The quotient of these two spaces
\[ H^1(\pi_1(X), \mathfrak{g}_{\agg \rho}) = \frac{Z^1(\pi_1(X), \mathfrak{g}_{\agg \rho})}{B^1(\pi_1(X), \mathfrak{g}_{\agg \rho})} \]
is called the \emph{space of infinitesimal deformations} of $\rho$. We say that $\rho$ is \emph{infinitesimally rigid} if such space is trivial.

This definition gains importance under the light of Weil's lemma \cite{We}, that asserts that any infinitesimally rigid representation is also \emph{locally rigid}, that is any deformation of $\rho$ is induced by a path of conjugations.

\begin{rem}
The cocycle condition implies the following:
$$z(g^{-1})= - \rho(g)^{-1} z(g) \rho (g), \qquad z(e) = 0.$$
\end{rem}

By construction we have a surjective homomorphism
$$\psi\colon \mathfrak g \longrightarrow B^1(\pi_1(X),\mathfrak{g}_{\agg \rho}), \qquad
\psi(V)( g) = V\rho(g) - \rho(g) V .$$
\begin{prop} \label{B:prop}
If the image of $\rho$ has limit set $S^{n-1} = \partial \matH^n$, the map $\psi$ is an isomorphism.
\end{prop}
\begin{proof}
If $\psi$ were not injective, we would get a non-trivial matrix $V\in \mathfrak g$ that commutes with the image of $\rho$. Therefore $e^V \in G = \Iso(\matH^n)$ would be a non-trivial isometry that commutes with the image of $\rho$. This is easily seen to be impossible when the limit set is the whole $S^{n-1}$.
\end{proof}

\begin{cor} \label{B:cor}
With the same hypothesis, we get 
$$\dim B^1(\pi_1(X), \mathfrak g_{\agg \rho}) = \dim \mathfrak g = \dim G = \frac{n(n+1)}{2}.$$
\end{cor}

\subsection{Finitely presented groups}
If $\pi_1(X)$ is finitely presented, we can determine
$H^1(\pi_1(X), \mathfrak{g}_{\agg \rho})$ as follows.
Given a finite presentation 
\[ \pi_1(X)= \langle\ g_1, \ldots, g_h\ |\ r_1, \ldots, r_k\ \rangle, \]
and a representation $ \rho\colon \pi_1(X) \to G$, 
a deformation of $\rho$ is of course determined by its behaviour on the generators. The same holds for a cocycle $z$ in $Z^1(\pi_1(X), \mathfrak{g}_{\agg \rho})$ since the following equalities hold
\[ z(g h) = z(g) + \rho(g) z(h) \rho (g)^{-1}, \qquad z(g^{-1})= - \rho(g)^{-1} z(g) \rho (g). \]
In particular $H^1(\pi_1(X), \mathfrak{g}_{\agg \rho})$ has finite dimension. 
An arbitrary assignment of elements in $\mathfrak{g}$ to the generators of $\pi_1(X)$ will not give rise to a cocycle in general: this assignment must fulfill some requirements that we now describe. 

We represent the cocycles using $\zeta$ instead of $z$. We get
\begin{equation} \label{zeta:eqn}
\zeta(gh) = \zeta(g) \rho(h) + \rho(g) \zeta(h), \qquad \zeta (g^{-1}) = - \rho(g)^{-1} \zeta(g) \rho(g)^{-1}.
\end{equation}
Given a word $w$ in $g_1,\ldots, g_h$ and their inverses and some invertible matrices $A_1,\ldots, A_h$, we denote by $w(A_1,\ldots,A_h)$ the matrix obtained by substituting in $w$ each $g_j$ with $A_j$. Consider a $h$-uple of matrices 
\[ D = (D_1, \ldots, D_h) \in (M(n+1,\matR))^h. \]

\begin{prop} \label{Dh:prop}
There is a cocyle $\zeta$ with $\zeta(g_i) = D_i$ for all $i$ if and only if:
\begin{itemize}
    \item every element $D_i$ is in the tangent space in $G$ at $M_i=\rho(g_i)$, and
    \item the relations vanish at first order along the direction $D$, \emph{i.e.}
\begin{equation} \label{ri:eqn}
      \frac{d}{dt}r_i(M_1+tD_1,\ldots,M_h+tD_h)\Big|_{t=0}=0 \ \ \forall i=1,\ldots,k. 
\end{equation}
\end{itemize}
In this case we have
\begin{equation} \label{zetag:eqn}
\zeta(g) = \frac{d}{dt}w(M_1+tD_1,\ldots,M_h+tD_h)\Big|_{t=0} 
\end{equation}
for every word $w$ that represents $g$.
\end{prop}

\begin{proof}
Let $D_1,\ldots, D_h$ satisfy both conditions. We define $\zeta(g)$ using (\ref{zetag:eqn}). The definition is well-posed: indeed, if for every word $w$ we define 
$$\zeta(w) = \frac{d}{dt}w(M_1+tD_1,\ldots,M_h+tD_h)\Big|_{t=0}$$
then we easily deduce from the Leibnitz rule that 
$$\zeta(w_1w_2) = \zeta(w_1) \rho(w_2) + \rho(w_1) \zeta(w_2), \qquad \zeta (w^{-1}) = - \rho(w)^{-1} \zeta(w) \rho(w)^{-1}.$$
By hypothesis $\zeta(r_i)=0$ and of course $\rho(r_i) = I$. This implies easily that $\zeta$ vanishes on every word obtained from the relators by conjugations, products, and inverses. This in turn easily implies that $\zeta(w_1) = \zeta(w_2)$ whenever two words $w_1$ and $w_2$ indicate the same element $g$ of the group. 

Conversely, let $\zeta$ be a cocycle. We first prove (\ref{zetag:eqn}). The equality holds when $w= g_i$, and we deduce easily that it holds for any word $w$ using (\ref{zeta:eqn}). We then deduce (\ref{ri:eqn}) using $w=r_i$.
\end{proof}

We can view a cocycle $\zeta$ as an assignment of matrices $D_1,\ldots, D_h$ to the generators $g_1,\ldots, g_h$ that fulfill some requirements. A coboundary is determined by a vector $V\in \mathfrak{g}$ as $\zeta(g) = V \rho(g) - \rho(g) V$. We may pick a basis for $\mathfrak{g}$ and get a finite set of generators for 
${B^1(\pi_1(X), \mathfrak{g}_{\agg \rho})}$.

\subsection{Fundamental groupoids}
The theory introduced in the previous pages extends from fundamental groups to fundamental groupoids with roughly no variation. This extension will be useful for us to prove Theorem \ref{rigidity_existence:teo}.

Let $X$ be a path-connected topological space and $V\subset X$ a finite set of points. The \emph{fundamental groupoid} $\pi_1(X,V)$ relative to $V$ is the 
the set of continuous maps $\alpha: [0,1] \to X$ with extremal points in $V$, up to homotopy which fixes the extremal points. It is possible to concatenate two such paths if the ending point of the first is the initial point of the second. When we write $\alpha_2 \alpha_1$ we mean that the first path is $\alpha_1$ and the second one is $\alpha_2$, and concatenation is possibile. 

The fundamental groupoid has a trivial element $e_v$ for every $v\in V$. For every $v\in V$, the inclusion defines an injection $i_*\colon \pi_1(X,v) \hookrightarrow \pi_1(X,V)$.

A \emph{finite presentation} of a groupoid is defined in the same way as for groups, as a set of generators and relators
$$\langle\ g_1, \ldots, g_h\ |\ r_1, \ldots, r_k \ \rangle$$
where each $g_i$ is an element of the groupoid, each $r_j$ is a word in the $g_i^{\pm 1}$ that represents some trivial element, every element of the groupoid is represented as a word in the $g_i^{\pm 1}$, and two words $w_1, w_2$ represent the same element if and only if $w_1w_2^{-1}$ makes sense and is obtained from the relations by formal conjugations, inversions, and multiplications.

Let $V = \{v_0,\ldots, v_s\}$. For every $i=1,\ldots, s$ pick an arc $\alpha_i$ connecting $v_0$ and $v_i$. Given a finite presentation $\langle\ g_i \ |\  r_j\ \rangle$ for $\pi_1(X,v_0)$, we can construct one for $\pi_1(X,V)$ by adding the arcs $\alpha_1,\ldots, \alpha_s$ as generators. 

We will always suppose that $\pi_1(X,V)$ has a finite presentation.

\subsection{Representations and cocycles}
A \emph{groupoid representation} is a map
\[ \sigma: \pi_1(X,V) \to G \]
such that $\sigma( \beta \alpha) = \sigma( \beta) \sigma(\alpha)$ whenever it is possible to concatenate $\alpha$ and $\beta$. (Usually, a groupoid representation assigns a vector space to each $v\in V$ and sends elements to morphisms between these vectors spaces: here we simply assign the same vector space $\matR^{n+1}$ to every $v$ and require the morphisms to lie in $G=O^+(n,1)$.)

Given a grupoid representation $\sigma$, it is possible to define its deformations and the cocycles as we did in the previous section, the only difference being that multiplications should be considered only when they make sense. 

A cocycle $z$ is a map $z\colon \pi_1(X,V) \to \mathfrak{g}$ that fulfills the cocycle condition (\ref{cocycle:eqn}) for every pair $g,h \in \pi_1(X,V)$ of elements that can be multiplied. We denote the vector space of all cocycles as $Z^1(\pi_1(X,V), \mathfrak{g}_{\agg \sigma})$. We still have two versions $z$ and $\zeta$ of the same cocycle that differ only by right multiplication by $\sigma(g)^{-1}$. As in the previous case a coboundary is determined by a vector $V\in \mathfrak{g}$ as $\zeta(g) = V \sigma(g) - \sigma(g) V$ and the subspace of all coboundaries is denoted by $B^1(\pi_1(X,V), \mathfrak{g}_{\agg \sigma})$. The quotient of these two spaces is $H^1(\pi_1(X,V), \mathfrak{g}_{\agg \sigma})$.

Proposition \ref{Dh:prop} is still valid in this context, with the same proof. If we have a finite presentation of $\pi_1(X,V)$, we can determine all the cocycles in $H^1(\pi_1(X,V), \mathfrak{g}_{\agg \sigma})$ in their $\zeta$ version by assigning some matrices $D_i$ to the generators that fulfill the indicated requirements at the relators.

The representation $\sigma$ induces a representation $\rho = \sigma \circ i_*$ for $\pi_1(X,v_0)$. The spaces $Z^1(\pi_1(X,V), \mathfrak{g}_{\agg \sigma})$ and $Z^1(\pi_1(X,v_0), \mathfrak{g}_{\agg \rho})$ are related in a simple way: 

\begin{prop}\label{rel_gruppoide_gruppo:prop}
The inclusion $i$ induces a surjective map
\[  i^* \colon Z^1(\pi_1(X,V), \mathfrak{g}_{\agg \sigma}) \twoheadrightarrow Z^1(\pi_1(X,v_0), \mathfrak{g}_{\agg \rho}). \]
The dimension of its kernel is $\dim (G) \cdot (|V|-1)$.
\end{prop}
\begin{proof}
Pick a finite presentation $\langle\ g_i\ |\ r_j\ \rangle$ for $\pi_1(X,v_0)$ and some arcs $\alpha_k$ connecting $v_0$ to $v_k$. Then $\langle\ g_i, \alpha_k\ |\ r_j \ \rangle$ is a presentation for $\pi_1(X,V)$. A cocycle $\zeta$ for $\rho$ extends to a cocycle $\zeta'$ for $\sigma$ by assigning to each $\alpha_k$ and arbitrary matrix tangent to $G$ in $\sigma(\alpha_k)$. The resulting $\zeta'$ is a cocycle by Proposition \ref{Dh:prop}, since both presentations have the same relators. There are $|V|-1$ arcs $\alpha_k$ and a space of dimension $\dim (G)$ to choose from for each arc.
\end{proof}

The previous proposition can be upgraded for the representation $\rho$ we are interested in. 

\begin{cor}\label{rel_gruppoide_gruppo_2:cor}
If the image of $\rho$ has limit set $S^{n-1}$, 
the inclusion $i$ induces a surjective homomorphism
\[  i^* \colon H^1(\pi_1(X,V), \mathfrak{g}_{\agg \sigma}) \longrightarrow H^1(\pi_1(X,v_0), \mathfrak{g}_{\agg \rho}). \]
The dimension of its kernel is $\dim (G) \cdot (|V|-1)$.
\end{cor}
\begin{proof}
The spaces $B^1(\pi_1(X,V), \mathfrak{g}_{\agg \sigma})$ and $B^1(\pi_1(X,v_0), \mathfrak{g}_{\agg \rho})$ have both dimension equal to $\dim G$ by Proposition \ref{B:prop}. Hence the map $i^*$ sends the former to the latter isomorphically.
\end{proof}

\subsection{Cube complexes}
Let $C$ be a finite connected cube complex. It is natural here to consider its set of vertices $V$ and the fundamental groupoid $\pi_1(C,V)$. This has a natural presentation
\begin{equation} \label{pres:eqn}
\langle\ g_1,\ldots, g_h\ |\  r_1,\ldots, r_k\ \rangle
\end{equation}
where $g_1,\ldots, g_h$ are the edges of $C$, oriented arbitrarily, and $r_1,\ldots, r_k$ are 4-letters words in the $g_i^{\pm 1}$ arising from the square faces of $C$, oriented arbitrarily. 

Let $v$ be a fixed vertex of $C$. To pass from the presentation (\ref{pres:eqn}) of $\pi_1(C,V)$ to one for the fundamental group $\pi_1(C,v)$ it suffices to choose a maximal tree $T$ in the 1-skeleton of $C$ and to add a relator $g_i$ representing every edge contained in $T$.

\subsection{Proving the rigidity of $\widetilde{M}$}
Here we explain the main method to prove infinitesimal rigidity.
The holonomy of $\widetilde{M}$ is a homomorphism
\[ \rho: \pi_1(\widetilde{M}) \longrightarrow \Isom(\matH^n) .\]
The manifold $\widetilde{M}$ deformation retracts onto the infinite cube complex $\widetilde{C}$.
The fundamental group $\pi_1(\widetilde{M})$ could be not finitely presented (this is the case in dimension 4, see \cite{BM}) and $\widetilde{C}$ is infinite, so the techniques introduced in the previous section do not apply here. However, it will be sufficient to consider a finite portion $C'$ of $\widetilde{C}$ and use the following.
\begin{prop} \label{teorema_subcom_finito}
Let $C'$ be a finite subcomplex of $\widetilde{C}$ such that
\[ i_*:\pi_1(C', v_0) \twoheadrightarrow \pi_1(\widetilde{C}, v_0) \]
is surjective. 
If $ \rho' = \rho \circ i_* $ is infinitesimally rigid then $\rho$ is infinitesimally rigid.
\end{prop}
\begin{proof}
The surjective homomorphism $i_*$ induces an injective homomorphism 
$$i^* \colon H^1(\pi_1(\widetilde{C},v_0), \mathfrak g_{\agg \rho}) \longrightarrow H^1(\pi_1(C',v_0), \mathfrak g_{\agg \rho\circ i_*}).$$
If the target space is trivial, the domain also is.
\end{proof}

Recall that every $k$-cube of $\widetilde{C}$ has some level. The finite subcomplex $F_{[m,n]}$ defined in Subsection \ref{Fnm:subsection} is connected, and by Corollary \ref{Fnmsurj:cor} the inclusion $F_{[m,n]} \rightarrow \widetilde{C}$ is $\pi_1$-surjective (we suppose that $n-m \geq 2$). In the cases described in Section \ref{dim_4:section}-\ref{dim_5:section} we will prove that the representation $\rho' = \rho \circ i_*$ is infinitesimally rigid.

Let $V$ (respectively, $V'$) be the set of vertices of $\widetilde{C}$ (respecively, $F_{[m,n]}$). The representation $\rho$ (respectively, $\rho'$) extends to a representation $\sigma$ (respectively, $\sigma'$) of the groupoid $\pi_1(\widetilde{C},V)$ (respectively, $\pi_1(F_{[m,n]},V')$) that is easy to describe.  
Consider the orbifold-covering
\[ p: \widetilde{M} \longrightarrow P \]
where $P$ is the polytope used in the construction of the manifold, interpreted as an orbifold $P = \matH^n/_\Gamma$. Here $\Gamma$ is the Coxeter group generated by reflections along the facets of $P$. The map $p$ sends every vertex of $V$ to the center $v$ of $P$.
It induces a map
\[ p_*: \pi_1(\widetilde{C}, V) \longrightarrow \pi_1^{orb}(P, v) = \Gamma. \]
Consider the natural presentation $\langle \ g_i\ | \ r_j\ \rangle$ of $\pi_1(\widetilde{C},V)$, with generators and relators corresponding to oriented edges and squares of $\widetilde{C}$. Every $g_i$ corresponds to an oriented edge of $\widetilde{C}$, which in turn determines a facet $F$ of $P$. The map $p_*$ sends $g_i$ to the reflection $r_F$ along $F$. The orientation of $g_i$ is not important since $r_F^2 = \id$.

The map $p_*$ is very convenient because it sends every generator $g_i$ to a reflection $r_F$. We write $\sigma = p_*$ and denote by $\sigma'$ its restriction to $\pi_1(F_{[m,n]},V')$.

We now need to calculate the dimension of $H^1(\pi_1(F_{[m,n]},V'), \mathfrak{g}_{\agg \sigma'})$. To do this we create a linear system and we study its solutions using MATLAB.

\section{The numerical analysis}\label{numerical:section}

Here we describe in detail the methods used and the results obtained.

\subsection{The algorithm}
In this subsection we explain the algorithm that we used to build up and solve our linear system. The code can be found at \cite{code}. To understand the numbers involved in the computation, we give names to certain quantities:

\begin{itemize}
    \item we call $n$ the dimension of the ambient space $\matH^n$ of the polytope $P$;
    \item we call $n_{\text{facets}}$ the number of facets of the polytope $P$;
    \item we call $n_{\text{2-cofaces}}$ the number of faces of codimension 2 of $P$;
    \item we call $c$ the number of colours used.
\end{itemize}

\subsubsection{Getting the combinatorics of $F_{[m,n]}$}
We choose to work with the subcomplex $F_{[-1,2s-1]}$, where $s$ is a positive integer. As $s$ grows, we are considering a bigger part of $\widetilde{C}$. Keep in mind that the deck transformation acts on $\widetilde{C}$ translating levels by $\pm 2$.

The vertices of $F_{[-1,2s-1]}$ are $2^{c-1} \cdot (2s+1)$ points.

We have written a Python code that enumerates the edges of the cube complex $C$ associated to $M$. As we already pointed out, these edges are in correspondence with the facets of the polytopes $P_v$ where $v$ is even.
The edges of $F_{[-1,2s-1]}$ will be $s$ copies of the edges of $C$: each edge of $C$ has a unique lift with one vertex of level $2i$ for $i=0,\ldots,s-1$. The number of edges of $F_{[-1,2s-1]}$ will be
\[ n_{\text{facets}} \cdot 2^{c-1} \cdot s. \]

The edges of $F_{[-1,2s-1]}$ can be used to provide a set of generators of the fundamental groupoid $\pi_1(F_{[-1,2s-1]},V')$. To promote a list of edges to a set of generators we need to orient each one of them. This orientation is used to interpret edges as paths, and should not be confused with the one given by the state of the polytope. We orient each edge from the vertex of even level towards the vertex of odd level.
Each generator $g_i$ is sent by $\sigma'$ to a reflection $r_F$ along a facet $F$ of the polytope $P$, that we denote by $M_i$ to be consistent with Proposition \ref{Dh:prop}.

Then we need to encode the squares of $F_{[-1,2s-1]}$. We start by getting a list of the squares of $C$. Then for each square we consider the lifts that have vertices with level in $[-1,2s-1]$. See Figures \ref{Hyperbolic_Octahedron:fig} - \ref{sqares_oct:fig} for an example.

\begin{figure}
 \begin{center}
  \includegraphics[width = 4cm]{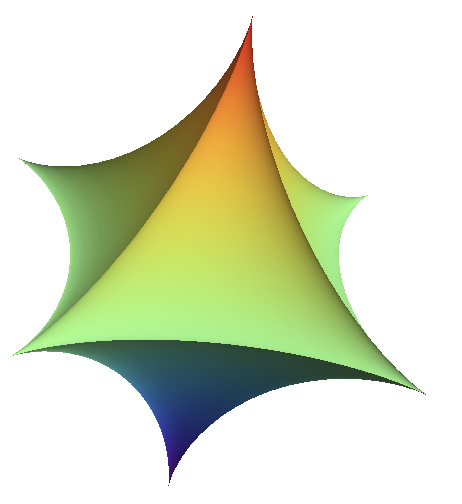}
 \end{center}
 \nota{The right-angled hyperbolic regular octahedron. Figure downloaded from \href{https://commons.wikimedia.org/wiki/File:Hyperbolic_Octahedron.jpg}{Wikimedia} and licensed under the \href{https://creativecommons.org/licenses/by-sa/3.0/deed.en}{Creative Commons Attribution-Share Alike}.}
 \label{Hyperbolic_Octahedron:fig}
\end{figure}

\begin{figure}
 \begin{center}
  \includegraphics[width = 8cm]{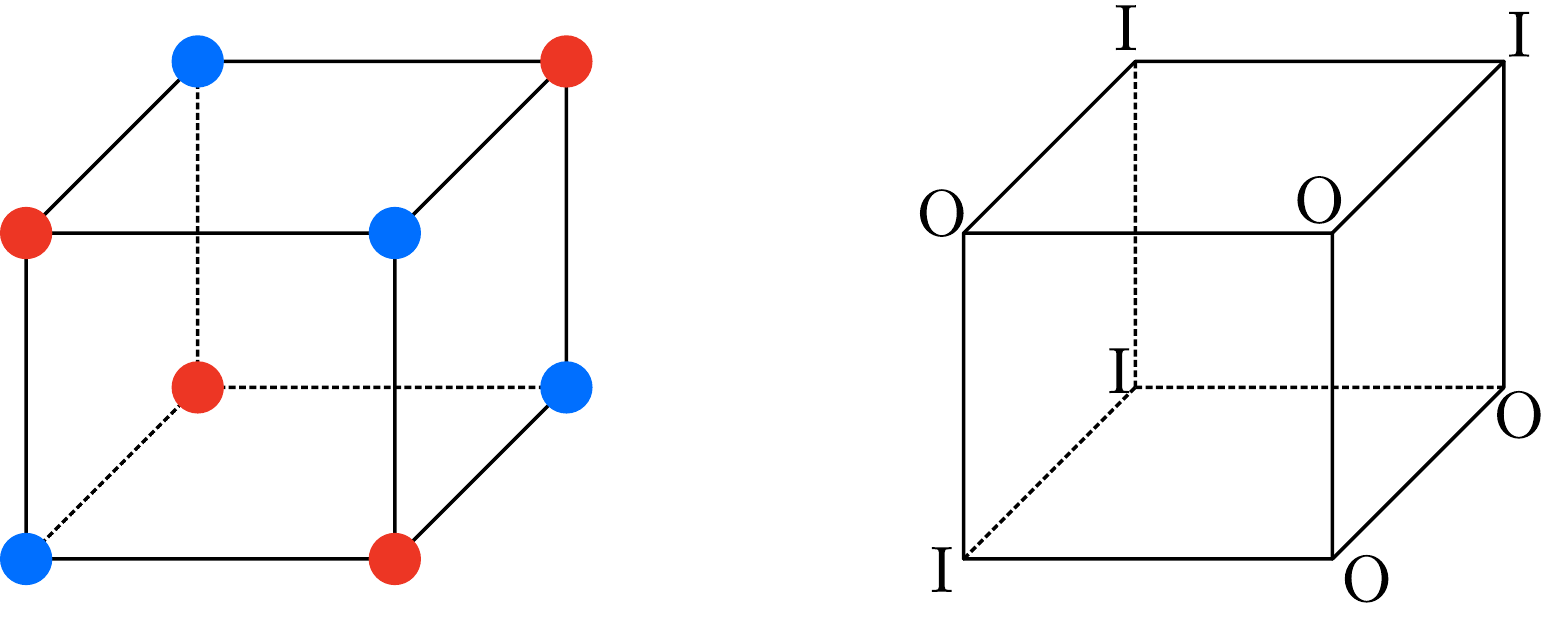}
 \end{center}
 \nota{The colouring and the state we consider on the octahedron, already described in \cite{JNW}. The cube depicted is the dual of the boundary of the octahedron, hence vertices correspond to facets and edges correspond to edges. There are no faces of dimension 2 because the vertices of the octahedron are ideal.}
 \label{colouring_and_state_oct:fig}
\end{figure}

\begin{figure}
 \begin{center}
  \includegraphics[width = 5cm]{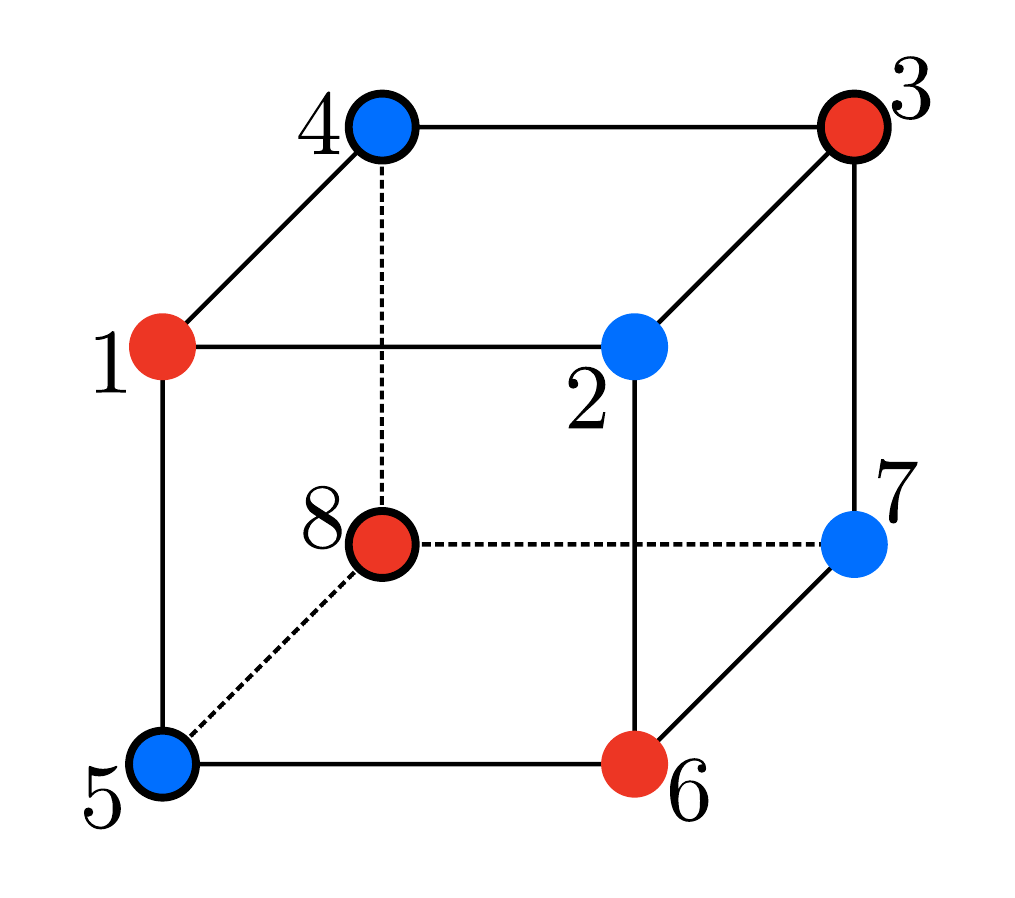}
 \end{center}
 \nota{One way to encode at the same time the colouring and the state. The contour of a vertex is enhanced (resp. not enhanced) if the corresponding facet has letter I (resp. O) in the state. We also enumerate the facets. We need this in order to list all the edges of the cube complexes $C$ and $\Tilde{C}$.}
 \label{colostate_oct:fig}
\end{figure}

\begin{figure}
 \begin{center}
  \includegraphics[width = 8cm]{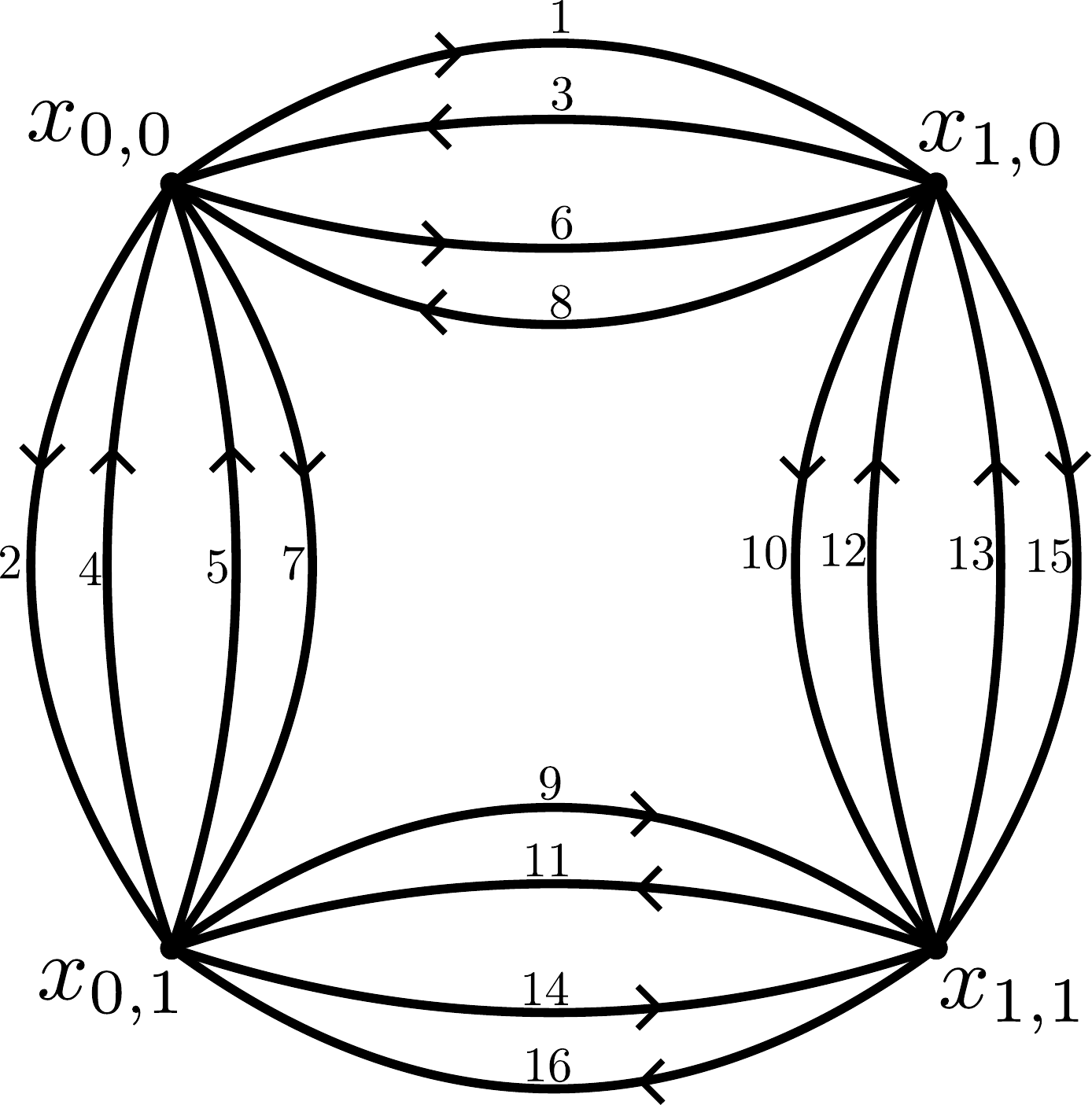}
 \end{center}
 \nota{The 1-skeleton of the cube complex $C$. The orientation on the edges is the one given by the state. In order to get the full $C$ we should add 12 squares. The enumeration of the edges is obtained in the following way: we enumerate the facets of $P$ (see Figure \ref{colostate_oct:fig}). The edges coming from the vertex $x_{0}$ are in correspondence with these facets. We then enumerate the vertices $x_{v}$ with $v$ even (in this case only $x_{1,1}$) and extend the enumeration on the edges coming from them.}
 \label{cube_compl_oct_M:fig}
\end{figure}

\begin{figure}
 \begin{center}
  \includegraphics[width = 9cm]{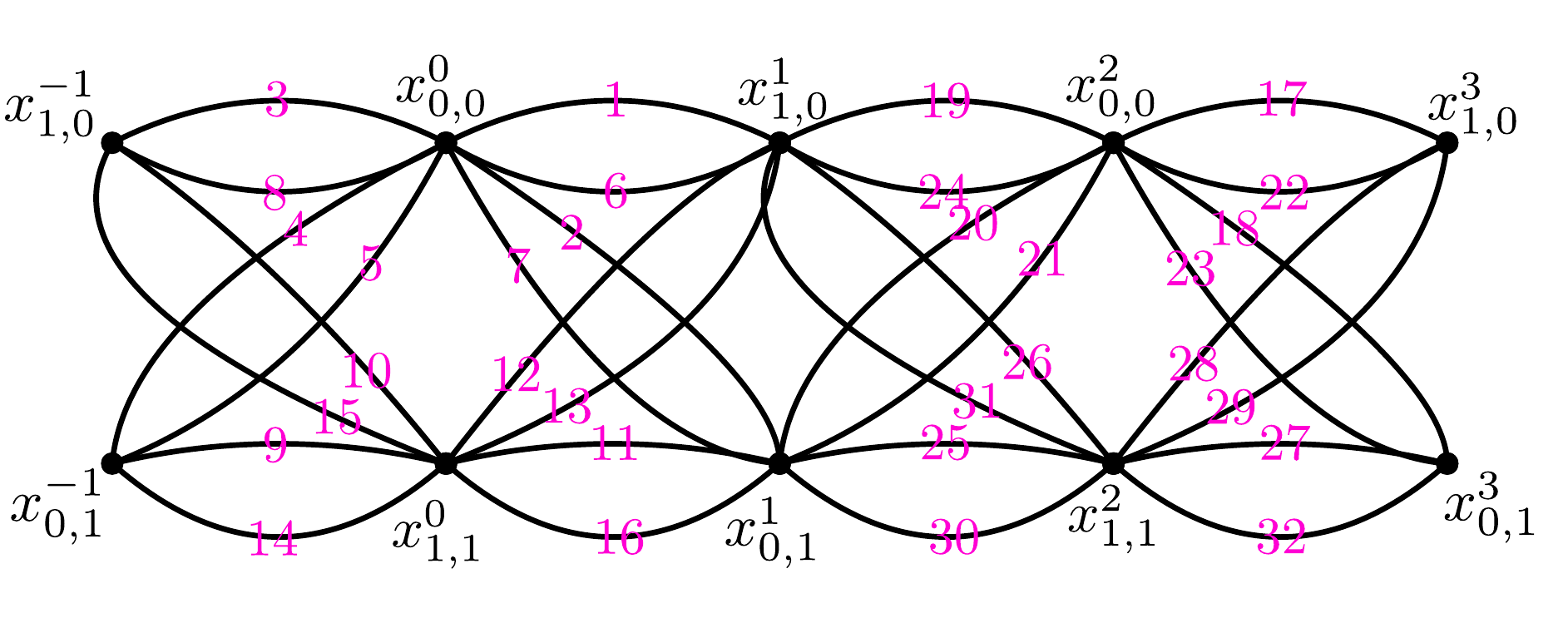}
 \end{center}
 \nota{The 1-skeleton of $F_{[-1,3]}$. The enumeration of edges follows the one of Figure \ref{cube_compl_oct_M:fig}: the edges with one vertex of level zero are in correspondence with the ones of $C$, and there is a copy of each of them with one vertex at level 2.}
 \label{F-13:fig}
\end{figure}

\begin{figure}
 \begin{center}
  \includegraphics[width = 8cm]{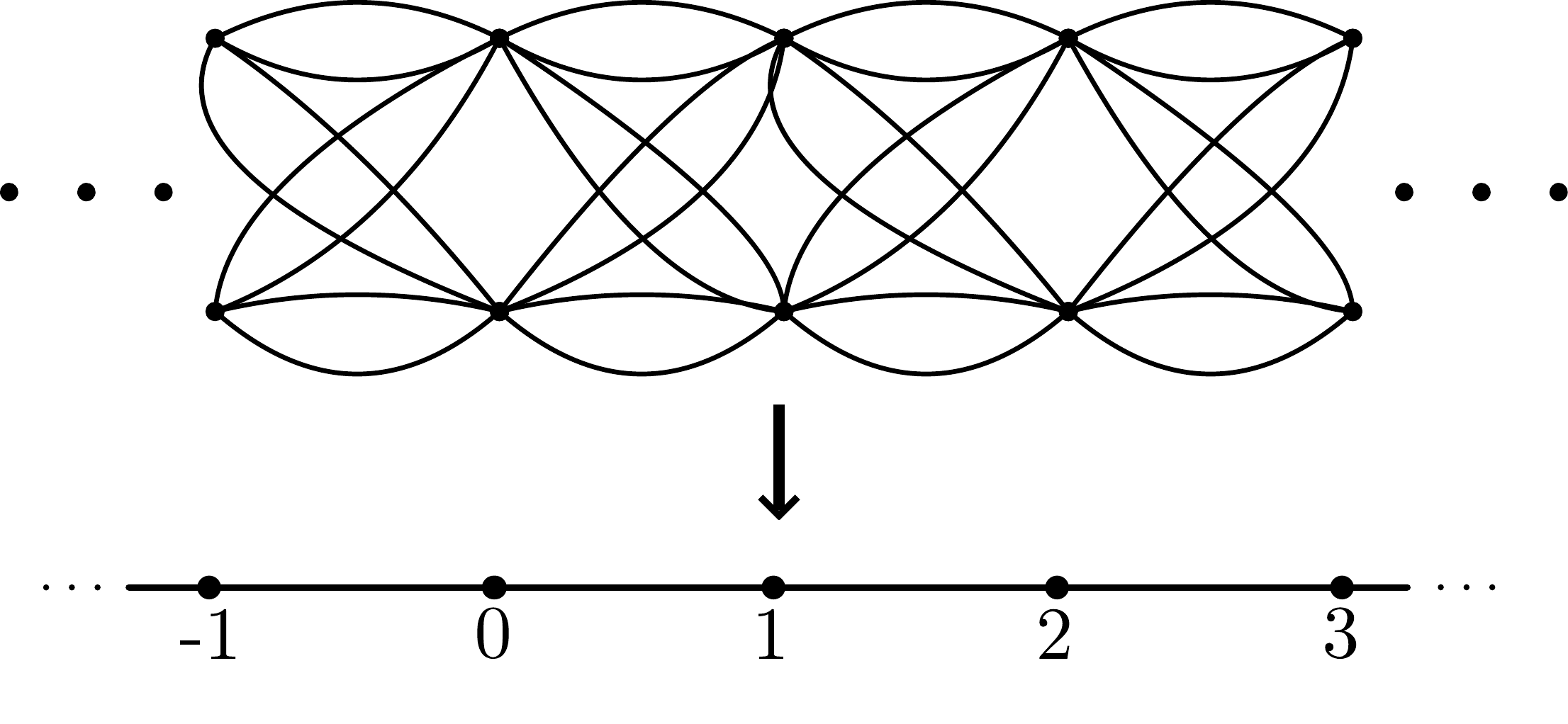}
 \end{center}
 \nota{One portion of the infinite 1-skeleton of the cube complex $\Tilde{C}$.}
 \label{cube_compl_oct_C_tilde:fig}
\end{figure}

\begin{figure}
 \begin{center}
  \includegraphics[width = 10cm]{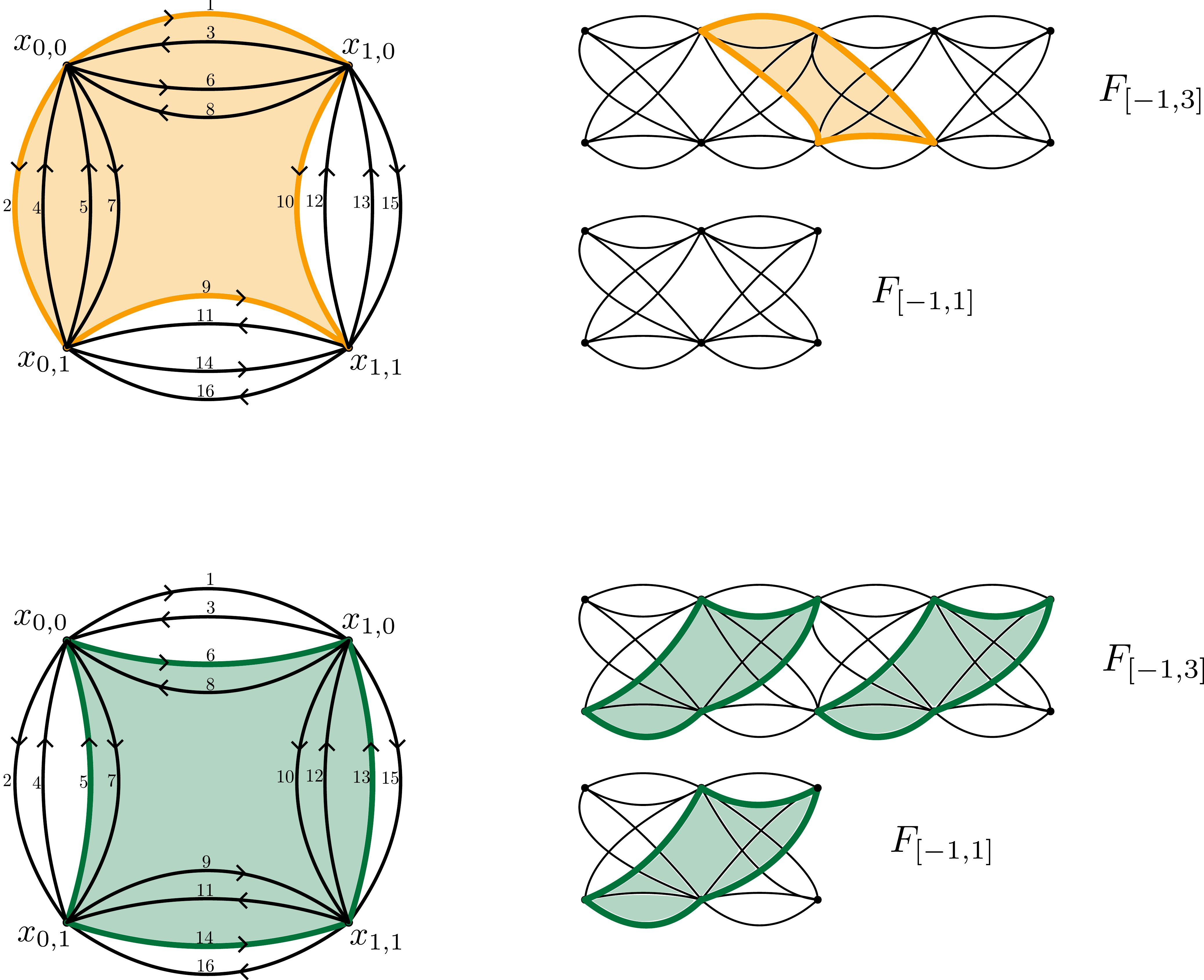}
 \end{center}
 \nota{Two squares and their lifts in $F_{[-1,1]}$ and $F_{[-1,3]}$.}
 \label{sqares_oct:fig}
\end{figure}

\subsubsection{Creating the linear system}\label{System:subsubsection}

We can now write the linear system with MATLAB following Proposition \ref{Dh:prop}. The dimension of the kernel of this linear system will be the dimension of $Z^1(\pi_1(F_{[m,n]},V'), \mathfrak{g}_{\agg \sigma'})$. We have a deformation $D_i$ for every generator $g_i$, and this gives $$ n_{\text{facets}} \cdot 2^{c-1} \cdot s \cdot (n+1)^2 $$ variables, since each $D_i$ is an unknown matrix $(n+1, n+1)$. 
\\
\\
\\
\\

Following Proposition \ref{Dh:prop}, the deformations have to fulfill two types of linear equations:
\begin{itemize}
    \item they must lie in the tangent space at $M_i$. This is achieved by imposing the equation
    \[ D_i^t \cdot J\cdot M_i + (J\cdot M_i)^t \cdot D_i=0, \]
    where $J$ is the diagonal matrix $(1,\ldots,1,-1)$;
    \item they have to solve the equations given by the relations. We have one equation for each square in $F_{[-1,2s-1]}$. There are a couple of tricks that can be used to simplify this type of equations; we make an example to show them explicitly. Let $Q$ be a square with boundary $g_1g_2^{-1}g_3g_4^{-1}$. The corresponding equation will be:
    \[ D_1 M_2^{-1} M_3 M_4^{-1} + M_1 (-M_2^{-1}D_2M_2^{-1}) M_3 M_4^{-1} + \] \[ + M_1 M_2^{-1} D_3 M_4^{-1} +  M_1 M_2^{-1} M_3 (-M_4^{-1}D_4M_4^{-1}) = 0.\]
    Every generator is sent by $\sigma'$ to a reflection, hence $M_i=M_i^{-1}$. In this way we avoid matrix inversions that are computationally heavy and could add numerical noise to the problem.
    Furthermore, we know that opposite sides of one square always go to the same matrix $M_i$ (because they correspond to the same facet of $P$). We can then simplify the equation using $M_1=M_3$ and $M_2=M_4$. It is also true that $M_1M_2=M_2M_1$ because they correspond to adjacent facets of a right-angled polytope.
    In the end we obtain:
    \[ D_1 M_1 - M_1 M_2 D_2 M_1 + M_1 M_2 D_3 M_2 - D_4 M_2 = 0.\]
    These simplifications helped a lot in numerical computations, and they were possible only because we used the groupoid structure. They also allow us to make the symbolical computations run faster.
\end{itemize}
Notice that these are matrix equations, so each of them represents actually $(n+1)^2$ equations (one for every entry of the matrix). By virtue of Proposition \ref{strata_number:prop}, the number of squares in $C$ is \[ n_{\text{2-cofaces}} \cdot 2^{c-2}. \]  We approximate the number of squares of $C$ that have a lift in $F_{[-1,1]}$ with one half of this value (this value is often correct for symmetric reasons. In any case, it is a good approximation), see Figure \ref{sqares_oct:fig}. Hence, the number or squares of $F_{[-1,2s-1]}$ is approximately
\[  n_{\text{2-cofaces}} \cdot 2^{c-2} \cdot \tonde*{s-\frac{1}{2}}. \] 

The linear system that we obtain has size that is approximately
\[ (n+1)^2 \tonde*{ n_{\text{facets}}  2^{c-1}  s + n_{\text{2-cofaces}} 2^{c-2} \tonde*{s-\frac{1}{2}} } \times (n+1)^2 n_{\text{facets}}  2^{c-1}  s  . \]

\subsubsection{The MATLAB rank function}
Once we have built the linear system, we need to compute its rank. This can be done in two different ways in MATLAB.
\begin{itemize}
    \item If the matrix is a symbolic matrix, the rank is calculated in a rigorous way. This is more time-consuming and requires more RAM in order to be carried out.
    \item If the matrix is a double matrix (\emph{i.e.}\ its entries are numerical values in the double precision) the rank function does the following: it calculates the singular values of the matrix (see \cite{svd}) and counts the ones that are greater than a tolerance value (the standard tolerance used by MATLAB depends on the size and on the norm of the matrix). The result is not rigorous. Its reliability can be estimated using the gap between the singular values greater than the tolerance and the smaller ones. In particular, a measure of the reliability of the computation is given by $\sfrac{\Sigma_{\text{min}}}{\sigma_{\text{max}}}$, where $\Sigma_{\text{min}}$ is the smallest singular value greater than the tolerance and $\sigma_{\text{max}}$ is the greatest singular value smaller than the tolerance. When the gap is big, the result is reliable. In our cases this gap is always big enough to let us trust the result, as we will show in detail.
    The Singular Value Decomposition algorithm is more time consuming than some alternatives, but it is also the most reliable.
\end{itemize}
We had not the time and the resources to carry out all the computations in the symbolic form for all the examples that we had. However, if we focus on one specific case, we are probably able to compute it in a rigorous way. For the 5-dimensional example (that is probably the most interesting one) we needed to use ad-hoc simplifications of the linear system in order to carry out the computations symbolically.

\subsection{Applying the Algorithm}
Here we describe the results obtained, that prove Theorem \ref{rigidity_existence:teo}.

\subsubsection{An example in dimension 3: checking known results}
We start by using the algorithm on a colouring and a state on the right-angled ideal octahedron in dimension 3, where all the results can be checked using the theory of hyperbolic 3-manifolds.

The colouring and the state $s_{0}$ are shown in Figure \ref{colouring_and_state_oct:fig}. The manifold obtained from this colouring is the minimally twisted chain link with six components. Using the methods in \cite{BM}, it is easy to show that this state induces a fibration on $S^1$. The fiber $N$ of this map is a six-punctured sphere.

The infinite cyclic covering is diffeomorphic to the product $N \times \matR$, hence its fundamental group is simply
\[ \pi_1(\Tilde{M})= \spam{g_1,\ldots,g_5}. \]
Since there are no relations, the dimension of the space of the cocycles is simply $\dim(O(3,1))\cdot 5 = 30$.
Using Corollary \ref{B:cor}, we deduce that $\dim(B^1) = \dim(O(3,1)) = 6$. Hence, the dimension of the space of infinitesimal deformation of $\Tilde{M}$ is:
\[ \dim H^1(\pi_1(\Tilde{M}), \mathfrak{g}_{\agg \rho'}) = 30-6 = 24.\] 

We now use our algorithm to compute the same quantity.

We start by finding the holonomy of the octahedron. Then using Sage we obtain the combinatorics of $F_{[-1,1]}$. With this data we can run the algorithm and find the singular values of the linear system of size $352 \times 256$ defined as in Section \ref{System:subsubsection}. The singular values that we obtain are plotted in Figure \ref{svd_oct:fig}. Of the 256 singular values, 196 are greater than $10^{-1}$, and 60 are smaller than $10^{-14}$. The tolerance suggested by MATLAB is $10^{-12}$. The values that are greater than the tolerance are 11 order of magnitude greater and the ones smaller than the tolerance are 2 order of magnitude smaller. Notice also that the singular values smaller than the tolerance are close to the machine epsilon: they are only two order of magnitude greater. For these reasons, we consider this computation reliable.

\begin{figure}
 \begin{center}
  \includegraphics[width = 8cm]{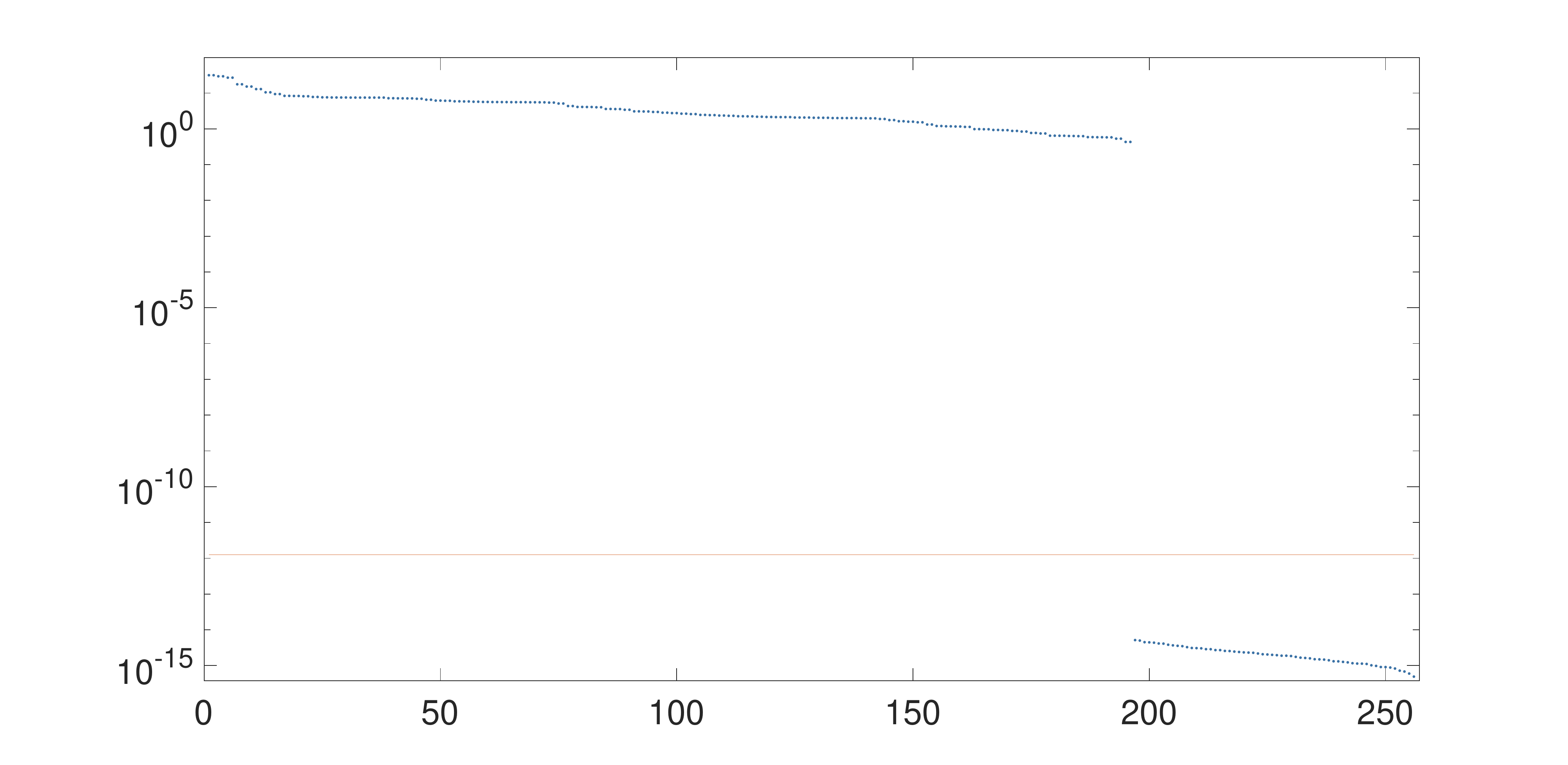}
 \end{center}
 \nota{In blue, the singular values of the linear system of the octahedron. In red, the tolerance suggested by MATLAB. The $y$-axis is in a logarithmic scale. The first 196 singular values are much greater than the tolerance, and the last 60 are much smaller than the tolerance. In every computation, the graph we get is very similar. }
 \label{svd_oct:fig}
\end{figure}

Once we get the value 60, we need to subtract the amount that comes from the fact that we are using the fundamental groupoid and the dimension of the coboundaries. Using Proposition \ref{rel_gruppoide_gruppo:prop} and Corollary \ref{B:cor} we estimate the dimension of the space of infinitesimal deformations with
\[ \dim(H^1(\pi_1(\widetilde{C},v_0)) \leq 60 - \#\set{\text{vertices of $F_{[-1,1]}$} }\cdot\dim(G)=60-6\cdot 6 = 24. \]

\begin{rem}
In this case the bound we find is sharp. However, with our algorithm we only get an upper bound for the dimension of the space of infinitesimal deformations.
This is because the map induced by the inclusion of $F_{[-1,1]}$ in $\Tilde{M}$ on the fundamental groups is surjective, but we do not know whether it is injective in general. However, when the upper bound is 0, we can conclude that $\Tilde{M}$ is infinitesimally rigid, as stated in Proposition \ref{teorema_subcom_finito}.
\end{rem}

\subsubsection{Dimension 4}\label{dim_4:section}

In dimension 4 we find new results using this machinery. Recall that we cannot have fibrations: this follows from the positiveness of the Euler characteristic. 

We start a right-angled polytope with a colouring and a state. We use the state to define an orientation on all edges of $C$ (in this section we follow the convention of \cite{BM}). In order to apply the algorithm we need all the ascending and descending links to be connected (recall Fact \ref{inclusion_pi1_surjective:fact}).

In \cite{BM} there are several examples that satisfy this condition. In particular we applied our algorithm in the following cases:
\begin{itemize}
    \item the polytope $P_4$ with the colouring and the state shown in \cite{BM}, Figures 4-5. In this case the manifold $M$ is the manifold $W$ described in \cite{BM}, Section 2.1;
    \item the $24$-cell with the colouring shown in \cite{BM}, Figure 6 and the 63 states described in \cite{BM}, Section 2.2. In this case the manifold $M$ is the manifold $X$ described in \cite{BM}, Section 2.2;
    \item the $120$-cell with the colouring and the state described in \cite{BM}, Section 2.4. In this case the manifold $M$ is the manifold $Z$ described in \cite{BM}, Section 2.4.
\end{itemize}
In every case we obtained the infinitesimal rigidity of the infinite cyclic covering. Every computations but one are numerical. With the 24-cell and one state (the most symmetric one, as described in \cite{BM}, Section 2.2.1 and shown in Figure \ref{24_cell_colostate:fig}), we carried out the computation symbolically. The results are shown in Tables \ref{C24:table}, \ref{C242:table}, and \ref{results:table}. In particular we proved:
\begin{teo}
The infinite cyclic covering of the manifold $X$ (described in \cite{BM}, Section 2.2) associated to the map induced by the status described in \cite{BM}, Section 2.2.1 is infinitesimally rigid.
\end{teo}

Notice that in all the cases considered in this section the function $f \colon C \to S^1$ can be extended on $M$ and is homotopic to a perfect circle-valued Morse function (see \cite{BM}).

\begin{table}
\begin{center}
\begin{tabular}{cc || cc || cc}
N & Volume & $\dim(H^1(F_{[-1,1]})$ & $\sfrac{\Sigma_{\text{min}}}{\sigma_{\text{max}}}$ & $\dim(H^1(F_{[-1,3]})$ & $\sfrac{\Sigma_{\text{min}}}{\sigma_{\text{max}}}$ \\
\hline
        1 & 194.869 & 0 & Symbolic & 0 & 7.235622e+12 \\ 
        2 & 189.473 & 0 & 2.397118e+12 & 0 & 6.148398e+12 \\ 
        3 & 186.874 & 0 & 2.335225e+12 & 0 & 6.532320e+12 \\ 
        4 & 186.34 & 1 & 3.543064e+12 & 0 & 7.089249e+12 \\ 
        5 & 185.307 & 0 & 3.202601e+12 & 0 & 6.263295e+12 \\ 
        6 & 185.035 & 0 & 3.660549e+12 & 0 & 4.937802e+12 \\ 
        7 & 184.813 & 0 & 2.220017e+12 & 0 & 4.920229e+12 \\ 
        8 & 184.301 & 7 & 5.933310e+12 & 0 & 7.657107e+12 \\ 
        9 & 184.067 & 3 & 4.489610e+12 & 0 & 5.853079e+12 \\ 
        10 & 183.873 & 0 & 2.429848e+12 & 0 & 5.587981e+12 \\ 
        11 & 183.867 & 0 & 3.748772e+12 & 0 & 5.742814e+12 \\ 
        12 & 183.544 & 2 & 2.986599e+12 & 0 & 6.617660e+12 \\ 
        13 & 183.437 & 1 & 3.461792e+12 & 0 & 6.452055e+12 \\ 
        14 & 183.393 & 0 & 2.532459e+12 & 0 & 5.506283e+12 \\ 
        15 & 183.122 & 0 & 3.923543e+12 & 0 & 7.137961e+12 \\ 
        16 & 182.36 & 1 & 3.219696e+12 & 0 & 5.422404e+12 \\ 
        17 & 182.281 & 1 & 3.659665e+12 & 0 & 4.927091e+12 \\ 
        18 & 182.171 & 0 & 2.146996e+12 & 0 & 4.765655e+12 \\ 
        19 & 181.283 & 0 & 3.882594e+12 & 0 & 4.002973e+12 \\ 
        20 & 181.127 & 0 & 3.311863e+12 & 0 & 4.097215e+12 \\ 
        21 & 181.025 & 0 & 2.291793e+12 & 0 & 4.734293e+12 \\ 
        22 & 180.934 & 1 & 2.509777e+12 & 0 & 5.412036e+12 \\ 
        23 & 180.825 & 2 & 1.236372e+12 & 0 & 6.850040e+12 \\ 
        24 & 180.661 & 1 & 3.918858e+12 & 0 & 6.296845e+12 \\ 
        25 & 180.451 & 1 & 4.446158e+12 & 0 & 6.131600e+12 \\ 
        26 & 180.387 & 0 & 1.875859e+12 & 0 & 3.726183e+12 \\ 
        27 & 180.331 & 1 & 3.638321e+12 & 0 & 6.280977e+12 \\ 
        28 & 180.248 & 0 & 3.738838e+12 & 0 & 6.589169e+12 \\ 
        29 & 180.128 & 0 & 2.809597e+12 & 0 & 4.884660e+12 \\ 
        30 & 179.869 & 0 & 2.775650e+12 & 0 & 5.124980e+12 \\ 
        31 & 179.754 & 0 & 2.232866e+12 & 0 & 5.923731e+12 \\ 
        32 & 179.657 & 1 & 2.477102e+12 & 0 & 4.884932e+12 \\

\end{tabular}
\vspace{.2 cm}
\nota{The numerical results on the 63 states on the 24-cell. Each state gives rise to a function $f \colon C \to S^1$. To distinguish these states we use the volume of the singular fiber $M^{\rm sing}$ described in \cite{BM}. The second column represents the volume of $M^{\rm sing}$. The third and the fourth are the dimension of the infinitesimal deformations calculated using MATLAB and the gap of the singular values of the system obtained using $F_{[-1,1]}$. The last two are the corresponding values obtained using $F_{[-1,3]}$. In the Tables 2-3 of \cite{BM} more information about $M^{\rm sing}$ can be found.}
\label{C24:table}
\end{center}
\end{table}

\begin{table}
\begin{center}
\begin{tabular}{cc || cc || cc}
N & Volume & $\dim(H^1(F_{[-1,1]})$ & $\sfrac{\Sigma_{\text{min}}}{\sigma_{\text{max}}}$ & $\dim(H^1(F_{[-1,3]})$ & $\sfrac{\Sigma_{\text{min}}}{\sigma_{\text{max}}}$ \\
\hline
        33 & 179.181 & 0 & 3.483970e+12 & 0 & 4.317760e+12 \\ 
        34 & 178.903 & 1 & 1.508899e+12 & 0 & 4.108744e+12 \\ 
        35 & 178.796 & 2 & 2.522961e+12 & 0 & 5.228315e+12 \\ 
        36 & 178.71 & 0 & 3.393137e+12 & 0 & 4.622471e+12 \\ 
        37 & 178.55 & 1 & 2.436939e+12 & 0 & 5.807914e+12 \\ 
        38 & 178.498 & 2 & 2.529432e+12 & 0 & 3.768823e+12 \\ 
        39 & 178.355 & 0 & 3.521648e+12 & 0 & 3.000208e+12 \\ 
        40 & 178.322 & 0 & 1.907534e+12 & 0 & 3.418561e+12 \\ 
        41 & 177.899 & 0 & 3.774676e+12 & 0 & 4.809060e+12 \\ 
        42 & 177.794 & 0 & 2.519382e+12 & 0 & 5.130856e+12 \\ 
        43 & 177.552 & 3 & 2.755276e+12 & 0 & 6.341695e+12 \\ 
        44 & 177.363 & 0 & 3.157805e+12 & 0 & 5.056077e+12 \\ 
        45 & 177.25 & 0 & 3.070059e+12 & 0 & 5.920372e+12 \\ 
        46 & 177.111 & 1 & 3.973227e+12 & 0 & 6.168683e+12 \\ 
        47 & 176.982 & 1 & 2.323239e+12 & 0 & 4.127369e+12 \\ 
        48 & 176.899 & 0 & 3.272597e+12 & 0 & 4.647325e+12 \\ 
        49 & 175.422 & 0 & 2.507132e+12 & 0 & 4.161510e+12 \\ 
        50 & 175.17 & 1 & 3.826100e+12 & 0 & 5.070140e+12 \\ 
        51 & 175.085 & 0 & 2.458223e+12 & 0 & 3.874525e+12 \\ 
        52 & 174.082 & 0 & 2.269447e+12 & 0 & 4.997747e+12 \\ 
        53 & 173.808 & 0 & 2.382119e+12 & 0 & 3.321352e+12 \\ 
        54 & 173.331 & 1 & 3.010625e+12 & 0 & 5.122391e+12 \\ 
        55 & 173.211 & 0 & 2.826802e+12 & 0 & 4.510913e+12 \\ 
        56 & 172.693 & 0 & 2.840796e+12 & 0 & 3.793823e+12 \\ 
        57 & 172.582 & 0 & 2.119408e+12 & 0 & 3.818880e+12 \\ 
        58 & 172.161 & 0 & 2.778001e+12 & 0 & 3.560645e+12 \\ 
        59 & 171.484 & 1 & 2.398577e+12 & 0 & 4.459881e+12 \\ 
        60 & 170.918 & 1 & 2.293547e+12 & 0 & 2.133791e+12 \\ 
        61 & 166.466 & 1 & 1.581402e+12 & 0 & 1.718246e+12 \\ 
        62 & 163.95 & 1 & 2.274303e+12 & 0 & 2.227807e+12 \\ 
        63 & 154.991 & 3 & 2.094852e+12 & 0 & 2.104681e+12 \\ 
        
\end{tabular}
\vspace{.2 cm}
\nota{The numerical results on the 63 states on the 24-cell. Each state gives rise to a function $f \colon C \to S^1$. To distinguish these states we use the volume of the singular fiber $M^{\rm sing}$ described in \cite{BM}. The second column represents the volume of $M^{\rm sing}$. The third and the fourth are the dimension of the infinitesimal deformations calculated using MATLAB and the gap of the singular values of the system obtained using $F_{[-1,1]}$. The last two are the corresponding values obtained using $F_{[-1,3]}$. In the Tables 2-3 of \cite{BM} more information about $M^{\rm sing}$ can be found (continue).}
\label{C242:table}
\end{center}
\end{table}

\begin{table}
\begin{center}
\begin{tabular}{c | c | cc }

Polytope & Size system & $\dim(H^1(F_{[-1,1]})$ & $\sfrac{\Sigma_{\text{min}}}{\sigma_{\text{max}}}$  \\

\hline

$P_4$ & $7000 \times 4000 $ & 0 & 1.2631e+12 \\
120-cell & $120000 \times 48000$ & 0 & 2.2308e+10 \\
$P_5$ & $ 175104 \times 73728 $ & 0 & Symbolic  \\

\end{tabular}
\vspace{.2 cm}
\nota{The results on the others right-angled polytopes. The first column describes the polytope. The second column represents the size of the linear system given to MATLAB, associated with $F_{[-1,1]}$. The third and the fourth ones contain the dimension of the infinitesimal deformations calculated using MATLAB and the gap of the singular values of the system.}
\label{results:table}
\end{center}
\end{table}

\begin{figure}
 \begin{center}
  \includegraphics[width = 8cm]{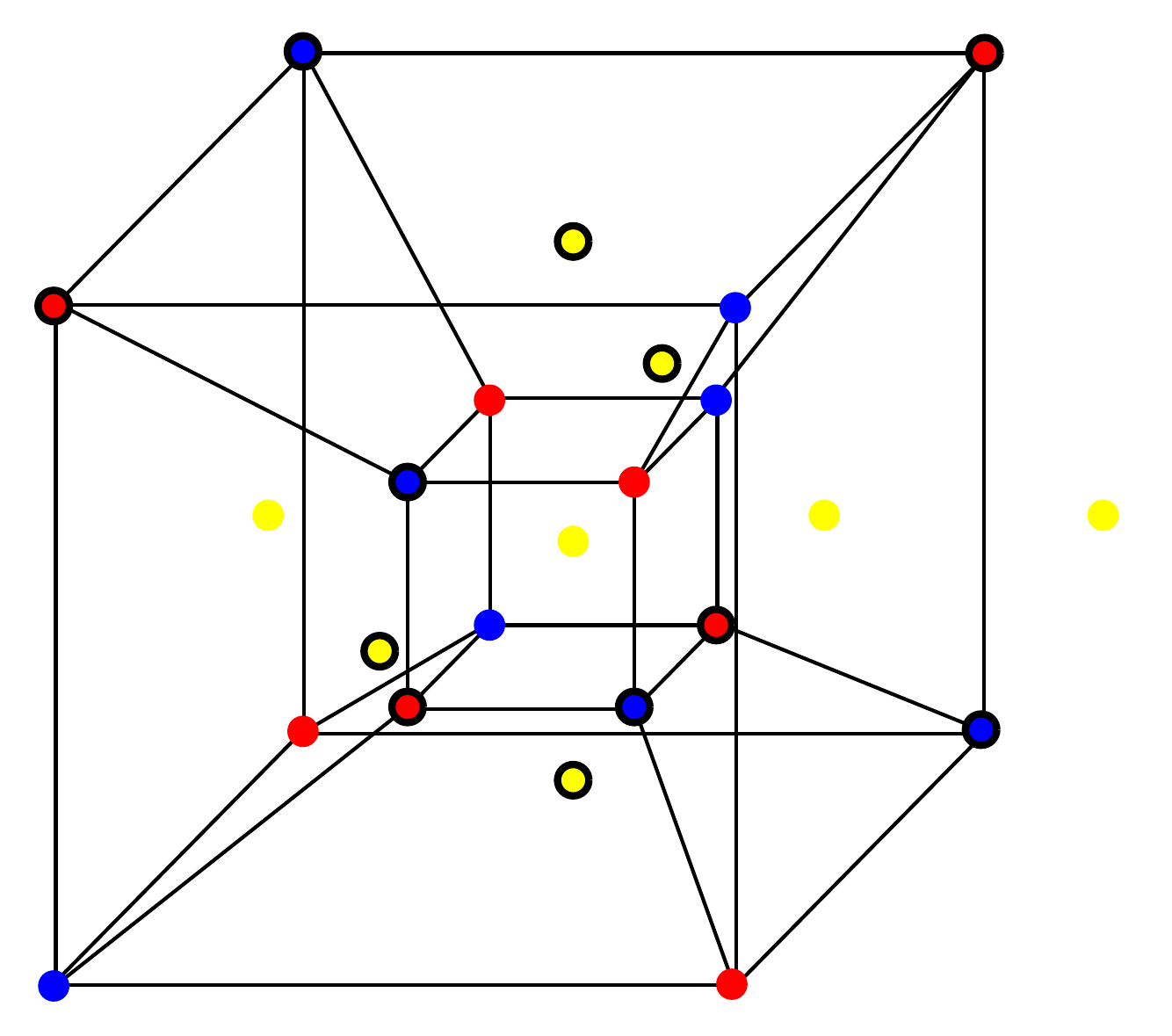}
 \end{center}
 \nota{The colouring and the state on the 24-cell for which we made out symbolic computations.}
 \label{24_cell_colostate:fig}
\end{figure}

\begin{figure}
 \begin{center}
  \includegraphics[width = 12cm]{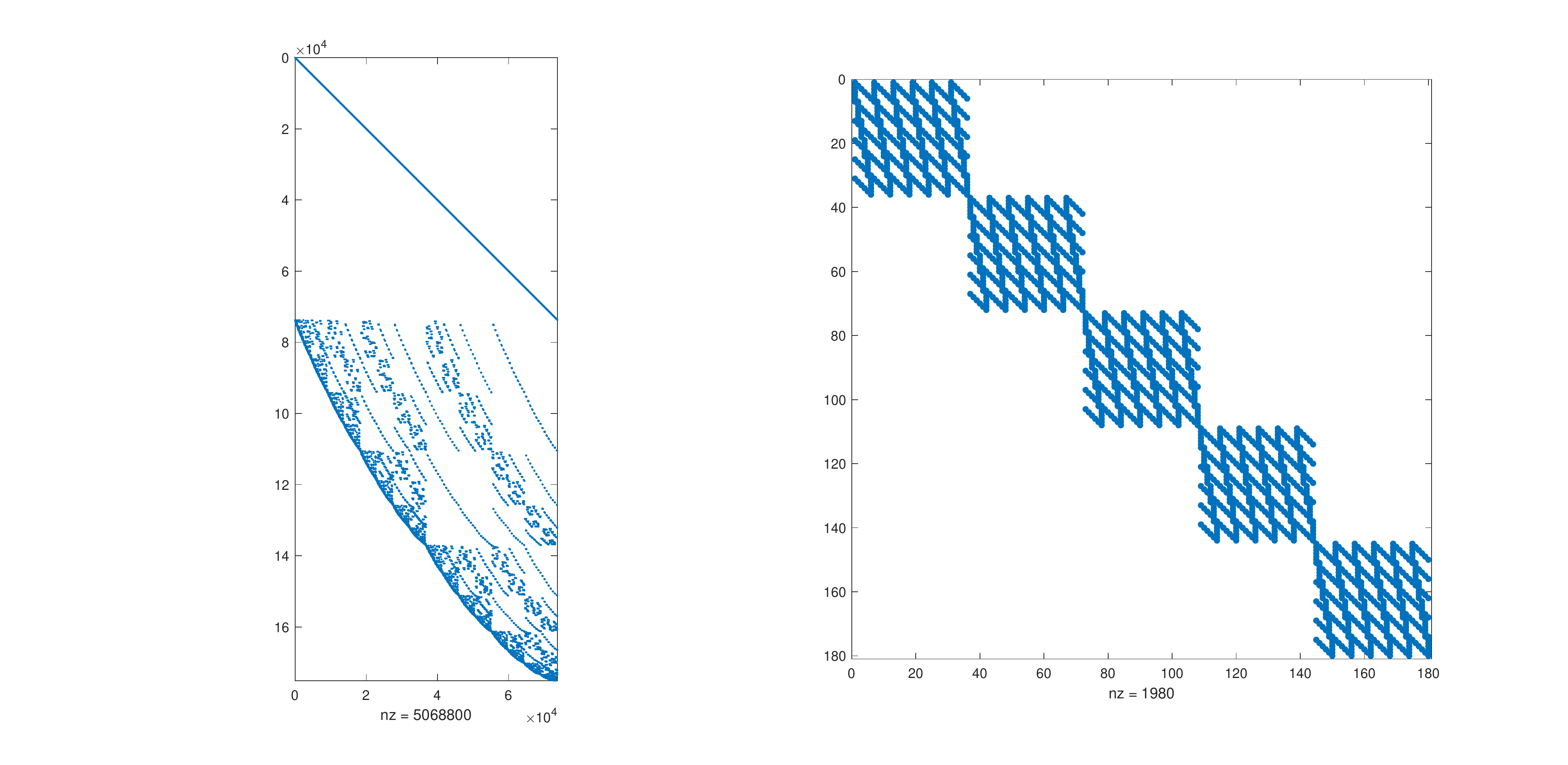}
 \end{center}
 \nota{The pictures obtained by using the MATLAB function \emph{spy} on the $ 175104 \times 73728 $ linear system coming from $P_5$ (left) and on its principal minor $180 \times 180$ (right). In the left picture we notice clearly the distinction between the equations given by the tangency conditions and the ones given by the squares. Given a $m \times n$ matrix $M$, the function \emph{spy($M$)} creates a $m \times n$ white grid and colours the element $(a,b)$ of the grid if and only if $M_{a,b}$ is not 0.}
 \label{Matrix_and_minor_spy:fig}
\end{figure}
\subsubsection{Dimension 5}\label{dim_5:section}
In \cite{itamarmig}, Italiano, Martelli, and Migliorini found an interesting example of fibration $f \colon M \to S^1$ in dimension 5. We can apply our algorithm to their construction to prove the infinitesimal rigidity of the associated infinite cyclic covering.

Following their construction, we use the polytope $P_5$, the paired colouring shown in \cite{itamarmig}, Figure 3 and the state shown in \cite{itamarmig}, Figure 9. To define the orientation on the edges of $C$ we use the convention in \cite{itamarmig}, Section 1.6.

The linear system associated with $F_{[-1,1]}$ has size $175104 \times 73728$. Trying to compute the rank of the symbolic linear system built as in Section \ref{System:subsubsection} made MATLAB freeze. In order to compute its exact rank, we had to simplify the system using its structure, see Figure \ref{Matrix_and_minor_spy:fig}. The manipulations derive from linear algebra considerations and the details can be found in \cite{code}. In particular we proved the following:
\begin{teo}
The infinite cyclic covering of the manifold $M^5$ associated to the map $f$ (both described in \cite{itamarmig}, Section 1) is infinitesimally rigid.
\end{teo}


\section{Related Results and Open Questions}\label{questions:section}

The results we found suggest several patterns that we discuss in this section.

\subsection{Ignoring relations}
It appears that it is often enough to use $F_{[-1,1]}$ to prove the rigidity of the manifold $\widetilde{M}$. The algorithm applied to this specific subcomplex can be interpreted in a nice way. Here we elaborate on this aspect.

We want to compare the algorithm applied to $F_{[-1,1]}$ with the algorithm applied to $C$, the cube complex on which the finite-volume manifold $M$ retracts. The number of vertices of $F_{[-1,1]}$ is $\frac32$ times the number of vertices of $C$: this is because the odd vertices have two lifts in $F_{[-1,1]}$ while the even vertices have only one lift. Let $V$ be the set of vertices of $C$ and $V'$ be the set of vertices of $F_{[-1,1]}$. The groupoids $\pi_1(C, V)$ and $\pi_1(F_{[-1,1]}, V')$ have the same number of generators: this holds because every edge in $C$ has exactly one lift in $F_{[-1,1]}$. If we look at the squares (that corresponds to relations in the groupoid), some squares of $C$ have one lift in $F_{[-1,1]}$ and some of them have zero lifts in $F_{[-1,1]}$, see Figure \ref{sqares_oct:fig}. The ones that have no lift in $F_{[-1,1]}$ are the ones whose lifts connect two even vertices that have different level. This means that the presentations of the groupoids $\pi_1(C, V)$ and $\pi_1(F_{[-1,1]}, V')$ differ only by a certain number of relations, that appear in $\pi_1(C, V)$ and do not appear in $\pi_1(F_{[-1,1]}, V')$.

When we build the linear system associated to $\pi_1(C, V)$, by the Mostow rigidity we know that the dimension of the kernel must be $\dim(G) \cdot \#V$: this is because the manifold $M$ has finite volume. With some states (the ones that made us able to prove the rigidity by looking at $F_{[-1,1]}$), ignoring the relations given by the squares that have no lift in $F_{[-1,1]}$ raised the dimension of the kernel to $\dim(G) \cdot \frac32 \cdot  \#V$. In other cases the kernel became greater (in these cases we needed to consider $F_{[-1,3]}$ to find rigidity, see Tables \ref{C24:table} and \ref{C242:table}).

\begin{quest}\label{behaviour_relations:quest}
Is there any nice way to distinguish between the states such that the complex $F_{[-1,1]}$ is enough to prove rigidity and the other ones?
\end{quest}

\subsection{Always rigid?}
In the papers \cite{BM,itamarmig} there are several examples where $\ker(f_*)$ is finitely generated. In some of these cases (the ones shown in Sections \ref{dim_4:section} and \ref{dim_5:section}) it was possible to apply our method, and we were always able to prove (or to obtain strong numerical evidence in favour of the fact) that the hyperbolic structure was infinitesimally rigid. Hence, it is quite natural to conjecture the following:

\begin{conj}\label{rigidity_always:conj}
Let $M$ be a finite volume hyperbolic manifold in dimension greater of equal than $4$. Let $f \colon M \to S^1$ be a non-homotopically trivial smooth map such that $\ker(f_*)$ is finitely generated, where $f_*$ is the map induced on the fundamental groups. Then the cyclic covering associated to the subgroup $\ker(f_*)$ is infinitesimally rigid.
\end{conj}

\bibliographystyle{plain}
\bibliography{references.bib}

\begin{thebibliography}{10}

\bibitem{svd}
Matlab help page on svd decomposition.
\newblock
  \emph{https://www.mathworks.com/help/matlab/math/singular-values.html}.

\bibitem{code}
Ludovico Battista.
\newblock Code for infinitesimal rigidity of cubulated manifolds.
\newblock \emph{https://people.dm.unipi.it/battista/code/irfcm}.
\newblock Released: 2021-12-20.

\bibitem{BM}
Ludovico Battista and Bruno Martelli.
\newblock Hyperbolic 4-manifolds with perfect circle-valued morse functions.
\newblock {\em To appear in Trans. Amer. Math. Soc.}, 2021.

\bibitem{BG}
Nicolas Bergeron and Tsachik Gelander.
\newblock A note on local rigidity.
\newblock {\em Geometriae Dedicata}, 107:111--131, 2004.

\bibitem{BB}
Mladen Bestvina and Noel Brady.
\newblock Morse theory and finiteness properties of groups.
\newblock {\em Inventiones mathematicae}, 129:445--470, 1997.

\bibitem{BrBr}
Jeffrey~F Brock and Kenneth~W Bromberg.
\newblock On the density of geometrically finite kleinian groups.
\newblock {\em Acta Mathematica}, 192(1):33--93, 2004.

\bibitem{B}
K.~Bromberg.
\newblock Projective structures with degenerate holonomy and the bers density
  conjecture.
\newblock {\em Annals of Mathematics}, 166(1):77--93, 2007.

\bibitem{CHK}
D.~Cooper, C.D. Hodgson, and S.~Kerckhoff.
\newblock {\em Three-dimensional Orbifolds and Cone-manifolds}.
\newblock MSJ Memoirs, Vol. 5. Mathematical Society of Japan, 2000.

\bibitem{FKS}
Leonardo Ferrari, Alexander Kolpakov, and Leone Slavich.
\newblock Cusps of hyperbolic 4-manifolds and rational homology spheres.
\newblock {\em To appear in Proceedings of the London Mathematical Society}.

\bibitem{itamarmig}
Giovanni Italiano, Bruno Martelli, and Matteo Migliorini.
\newblock Hyperbolic 5-manifolds that fiber over {$S^1$}, 2021.

\bibitem{JNW}
Kasia Jankiewicz, Sergey Norin, and Daniel~T. Wise.
\newblock Virtually fibering right-angled coxeter groups.
\newblock {\em Journal of the Institute of Mathematics of Jussieu},
  20(3):957–987, 2021.

\bibitem{KS}
Steven Kerckhoff and Peter Storm.
\newblock Local rigidity of hyperbolic manifolds with geodesic boundary.
\newblock {\em Journal of Topology}, 5, 2009.

\bibitem{KolSla}
Alexander Kolpakov and Leone Slavich.
\newblock {Hyperbolic 4-manifolds, colourings and mutations}.
\newblock {\em Proceedings of the London Mathematical Society},
  113(2):163--184, 2016.

\bibitem{libroM}
Bruno Martelli.
\newblock An introduction to geometric topology, 2016.

\bibitem{NS}
Hossein Namazi and Juan Souto.
\newblock {Non-realizability and ending laminations: Proof of the density
  conjecture}.
\newblock {\em Acta Mathematica}, 209(2):323 -- 395, 2012.

\bibitem{PLV}
Leonid Potyagailo and Ernest Vinberg.
\newblock On right-angled reflection groups in hyperbolic spaces.
\newblock {\em Commentarii Mathematici Helvetici}, 80:63--73, 2005.

\bibitem{We}
Andre Weil.
\newblock Remarks on the cohomology of groups.
\newblock {\em Annals of Mathematics}, 80(1):149--157, 1964.

\end{thebibliography}

\end{document}